\documentclass[11pt]{article}
\usepackage{amsfonts}
\usepackage{amssymb}
\usepackage{url,graphicx,tabularx,array,geometry}
\usepackage{amsmath,lipsum}

\usepackage{color}
\usepackage{epstopdf}
\usepackage{verbatim}
\usepackage{enumitem} 
\usepackage{multirow} 
\usepackage{float}
\usepackage{hyperref}
\usepackage{subfig}

\usepackage{amsthm} 

\newtheorem{theorem}{Theorem}[section]
\newtheorem{lemma}[theorem]{Lemma}
\newtheorem{proposition}[theorem]{Proposition}

\newtheorem{example}[theorem]{Example}
\theoremstyle{remark}

\newcommand{\vectornorm}[1]{\left\|#1\right\|}
\newcommand{\R}{\mathbb{R}}

\DeclareMathOperator{\diag}{diag}

\DeclareMathOperator{\trace}{trace}
\DeclareMathOperator{\tr}{tr}

\begin{document}
	
\title{On Solving the Quadratic Shortest Path Problem}
\author{Hao Hu \thanks{CentER, Department of Econometrics and OR, Tilburg University, The Netherlands, {\tt h.hu@uvt.nl}}
	\and {Renata Sotirov \thanks{Department of Econometrics and OR, Tilburg University, The Netherlands, {\tt r.sotirov@uvt.nl}}} }
\date{}

\maketitle

\begin{abstract}
The quadratic shortest path problem is the problem of finding a path in a directed graph such
that the sum of interaction costs over all pairs of arcs on the path is minimized.
We derive several semidefinite programming relaxations for the quadratic shortest path problem
with a matrix variable of order $m+1$, where $m$ is the number of arcs in the graph.
We use the  alternating direction method of multipliers to solve the semidefinite programming  relaxations.
Numerical results show that our bounds are currently  the strongest bounds for the quadratic shortest path problem.

We also present computational results on solving the quadratic shortest path problem using a branch and bound algorithm.
Our  algorithm computes a semidefinite programming bound in each node of the search tree,
and solves instances with up to 1300 arcs in less than an hour (!).
\end{abstract}


\section{Introduction}

The quadratic shortest path problem (QSPP) is the problem of finding a path in a directed graph  from the source vertex $s$ to the target vertex $t$
such that the sum of costs of arcs and the sum of interaction  costs over all distinct pairs of arcs on the path is minimized.
The QSPP is a NP-hard combinatorial optimization problem, see \cite{hu2016qspp,rostami2016quadratic}.  Rostami et al.~\cite{rostami2016quadratic} show that the problem remains NP-hard
even for the adjacent QSPP. That is a variant of the QSPP where the interaction costs of all non-adjacent arcs are equal to zero.
Hu and Sotirov~\cite{hu2016qspp} give an alternative proof for the same result using a  simple reduction from the arc-disjoint paths problem.

It is also known that the QSPP can be solved efficiently for particular families of graphs and/or for special cost matrices.
In particular, Rostami et al.~\cite{frey2015quadratic} provide a polynomial time algorithm for the adjacent QSPP considered on directed acyclic graphs.
Hu and Sotirov~\cite{hu2016qspp} show that the QSPP can be efficiently solved if the cost matrix is a non-negative symmetric product matrix,
or if the cost matrix is a sum matrix and every $s$-$t$ path in the graph has constant length.
In \cite{hu2016qspp}, it is also shown that the linearizability of the QSPP on grid graphs can be detected in polynomial-time.
We say that an instance of the QSPP is linearizable if its optimal solution can be found by solving the corresponding instance of the shortest path problem.
The algorithm from \cite{hu2016qspp} verifies whether a QSPP instance on  the $p \times q$ grid graph
is linearizable in $\mathcal{O}(p^{3}q^{2}+p^{2}q^{3})$ time, and if it is linearizable the algorithm returns the linearization vector.

Buchheim and Traversi \cite{buchheim2015quadratic} study separable underestimators that can be used to solve binary programs with a quadratic objective function. In particular, they provide an exact approach for the quadratic shortest path problem, which  solves instances on the $15 \times 15$ grid graph within one and a half hour.
Rostami et al.~\cite{rostami2016quadratic} present several lower bounding approaches for the QSPP, including a Glimore-Lawler
(GL) type bound  and a bound based on a reformulation scheme that iteratively improves the GL bound. We refer to the latter bound as RBB.
The numerical results in \cite{rostami2016quadratic} show that the branch-and-bound algorithm, which computes the RBB bound in each node of the tree,
provides an optimal solution for the QSPP with a  dense cost matrix   on the $15 \times 15$ grid graph within $96$ seconds.

The QSPP  arises in many different applications such as  route-planning problems in which the choice of a route is based on the mean as well as the variance of the path travel-time, see \cite{Sen:01}.
In \cite{NieWu, sivakumar1994variance}, the authors study several variants of the shortest path problem that are related to the QSPP, including the reliable shortest path problem
and a variance-constrained shortest path problem. The QSPP also plays a role in network protocols. In particular,  different restoration schemes of survivable asynchronous
transfer mode networks can be formulated as a QSPP, see Murakami and Kim \cite{murakami1997comparative}. Gourv{\`e}s et al.~\cite{gourves2009minimum}
consider the QSPP on undirected edge-colored graphs with non-negative reload costs.
The edge-colored graphs are for example used to model cargo transportation and large communication networks, see \cite{galbiati2008complexity, wirth2001reload}.
The QSPP can be also applied in satellite network designs as discussed in \cite{gamvros2006satellite}. \\

\noindent
\textbf{Main results and outline.} \\
In this paper we derive several semidefinite programming (SDP) relaxations with increasing complexity, for the quadratic shortest path problem.
The matrix variables in the SDP relaxations are of order $m+1$, where $m$ is the number of the arcs in the graph.
Our strongest SDP relaxation has a large number of constraints, and is difficult to solve  by an interior-point algorithm for instances of moderate size,
i.e., for graphs with more than 500 arcs.
Therefore, we implement the alternating direction method of multipliers (ADMM) to solve  the two  strongest semidefinite programming relaxations.
We adopt  the ADMM version of the algorithm suited for solving SDP relaxations  that was recently introduced by Oliveira, Wolkowicz and  Xu \cite{oliveira2015admm}.
The ADMM-based algorithm  computes our strongest SDP bound on a graph with 480 arcs in about one minute, while an interior-point algorithm needs 45 minutes.
The ADMM algorithm requires at most 46 minutes to compute the strongest SDP bound for an instance of the QSPP problem with 2646 arcs.

In order to incorporate the ADMM algorithm within  a branch-and-bound (B$\&$B) framework, we show how to improve the performance of the ADMM.
In particular, we improve its performance by projecting one of the variables onto a more intricate set than in the general settings.
This turns out to be the key to efficiently  obtain good bounds in each node of the  B$\&$B algorithm.
Our B$\&$B algorithm  finds an optimal solution for the QSPP on a grid graph with 760 arcs  in about three minutes.
We solve instances of the QSPP with 1300 arcs in less than an hour. On the other hand, Cplex can solve instances with less than $365$ arcs.

The paper is structured as follows.
In Section \ref{sect:probl_form}, we provide an integer programming formulation of the quadratic shortest path problem, and introduce several graphs that are used in our numerical tests.
In Section \ref{section_sdpLB}, we derive three semidefinite programming relaxations for the QSPP with increasing complexity.
Section \ref{SDP-Slater} provides the Slater feasible versions of the SDP relaxations.
In the same section we show how to obtain explicit expressions of the projection matrices corresponding to the relevant graphs.
In the case that the underlying graph is acyclic and/or  every $s$-$t$ path has the same length, feasible points in the SDP relaxations satisfy certain properties,
which we present  in Section \ref{sect:Special graphs}.
We outline the main features of the ADMM algorithm for the SDPs from \cite{oliveira2015admm} in  Section \ref{section_admm}.
Our tailored version of the ADMM algorithm is given in Section \ref{sect:FastADMM}.
Section \ref{sect:NumerExper} provides computational results on various instances.

\section{Problem formulation} \label{sect:probl_form}

Let $G = (V,A)$ be a directed graph with vertex set $V$, $|V|=n$, and arc set $A$, $|A|=m$.
A path is defined as an ordered set of vertices
${(v_{1},\ldots,v_{k})}$, $k> 1$ such that $(v_{i},v_{i+1})\in A$ for $i=1,\ldots,k-1$, and it does not contain repeated vertices. A $s$-$t$ path is a path ${P =(v_{1},v_{2},\ldots,v_{k})}$
such that $v_{1}$ is the source vertex ${s\in V}$ and $v_{k}$ is the target vertex $t\in V$.

A natural way to model the quadratic shortest path problem using  binary variables
is to represent a $s$-$t$ path $P$ by its characteristic vector $x$.
Thus, $x \in \{0,1\}^{m}$  and $x_{e} = 1$ if and only if the arc $e$ is in the path $P$.
Let  ${Q}=(q_{e,f}) \in \R^{m \times m}$ be a nonnegative symmetric matrix  whose rows and columns are indexed by the arcs.
The sum of the off-diagonal entries $q_{e,f} + q_{f,e}$ equals  the interaction cost between arcs $e$ and $f$, $e\neq f$.
The linear cost  of an arc $e$ is given by the diagonal element $q_{e,e}$ of the matrix $Q$.
Now, the quadratic cost of a path $P$ is given as follows:
\[
\sum_{e,f \in A, ~e\neq f} q_{e,f} x_{e}x_{f} + \sum_{e\in A} q_{e,e} x_{e} = x^{\mathrm T}Qx.
\]
Let us define the path polyhedron.
The incidence matrix ${\mathcal I}$ of $G$ is a $n \times m$ matrix that  has a row for each vertex and column for each arc,
such that ${\mathcal I}_{v,e} = 1$ if the arc $e$ leaves vertex $v$, $-1$ if it enters vertex $v$, and zero otherwise.
The $i$th row of the incidence matrix is denoted by $a_{i}^{\mathrm T}$ ($i=1,\ldots,n$).
Define the vector $b\in \R^{n}$ such that $b_{i} = 1$ if $i = s$, $-1$ if $i = t$, and zero otherwise.
Now, the path polyhedron $P_{st}(G)$ is given as follows:
\begin{equation}\label{stPathPolytope}
P_{st}(G) := \{ x\in \R^{m} \; | \;  0 \leq x \leq 1  ,\;\;  a_{i}^{\mathrm T}x= b_{i}, \;\; \forall i\in V \backslash \{t\}\}.
\end{equation}
Note that the constraint $a_{t}^{\mathrm T}x= b_{t}$ is not included in $P_{st}(G)$ as it is redundant.
It is a well-known result that the extreme-points of the polyhedron $P_{st}(G)$ correspond to the characteristic vectors of the $s$-$t$ paths.

The QSPP can be modeled as the following binary quadratic programming problem:
\begin{equation}\label{quadraticQSPP}
\begin{aligned}
& \text{minimize}
& & x^{\mathrm T}Qx \\
& \text{subject to}
& &  x \in P_{st}(G) & \hspace{0.5cm} &\\
&
& & x  \text{ binary}. & \hspace{0.5cm} &
\end{aligned}
\end{equation}

Clearly,  problem \eqref{quadraticQSPP} reduces to the linear shortest path problem if $Q$ is a diagonal matrix. We next provide several graphs that are used in the remainder of the paper.
\begin{example} \label{def:GridGraphs}{\em
The \emph{grid graph}  $G_{p,q} = (V,A)$ is a directed graph whose vertex and edge sets are given as follows:
\begin{align*}
V &= \{ v_{i,j} \;|\; 1 \leq i \leq p, \; 1 \leq j \leq q \}, \\
A & = \{ (v_{i,j},v_{i',j'}) \;|\; |i-i'|+|j-j'| = 1, \; i'\geq i, \; j' \geq j  \}.
\end{align*}
\noindent
Note that $|V| = pq$ and $|A| = 2p q-p-q$. Unless specified otherwise, we assume that the source vertex is $v_{1,1}$ and the target vertex is $v_{p,q}$.
Thus, all vertices except $v_{1,1}$ and $v_{p,q}$ are transshipment vertices. Every $s$-$t$ path in $G_{p,q}$ has the same length.}
\end{example}

\begin{example} \label{def:GridGraphs2}{\em
The \emph{flow grid graph} $G^f_{p,q}=(V,A)$
consists of transshipment vertices forming the $p \times q$ grid  as well as two extra vertices; a source vertex $s$ and a target vertex $t$.
Arcs between vertices on the grid are given as in Example \ref{def:GridGraphs}.
Additionally, there are $p$ arcs from $s$ to the vertices in the first column of the grid, and $p$ arcs from the last column of the grid to $t$.	
Note that there are  $pq+2$ vertices and $2pq+p-q$ arcs in $G^f_{p,q}$.}
\end{example}

\begin{example} 	\label{def:GridGraphs3} {\em
The \emph{double-directed grid graph} $\bar{G}_{p,q}=(V,A)$   has the same vertex set as the grid graph $G_{p,q}$.
The arc set of $\bar{G}_{p,q}$ is given as follows ${A  = \{ (v_{i,j},v_{i',j'}) \;|\; |i-i'|+|j-j'| = 1 \}.}$ Note that $|V| = pq$ and $|A| = 4pq-2p-2q$.}
\end{example}

\begin{example} \label{def:Kpartite} {\em
An \emph{incomplete $K$-partite graph} $G_{K}=(V,A)$  is a directed graph whose  vertices are partitioned into $K$ disjoint sets $V_1$, \ldots, $V_K$,
such that no two vertices within the same set are adjacent,  and every  vertex in $V_i$ is adjecent to every vertex in $V_{i+1}$ ($i=1,\ldots, K-1$).
In particular, we have that $(u,v)\in A$ for $u\in V_i$ and $v\in V_{i+1}$ where  $i=1,\ldots, K-1$.
}
\end{example}

\section{SDP relaxations for the QSPP} \label{section_sdpLB}

In this section, we derive three SDP relaxations for the QSPP with increasing complexity.
Our strongest relaxation has $m+n$ equalities and ${m \choose 2}$ non-negativity constraints.

In order to derive an SDP relaxation for the QSPP, we linearize the objective function $\trace(x^{\mathrm T}Qx) = \trace(Qxx^{\mathrm T})$ by replacing $xx^{\mathrm T}$ by a new variable $X \in \mathcal{S}^{m}$. Here, $\mathcal{S}^{m}$ denotes the set of symmetric matrices of order $m$. Clearly, for ${x \in P_{st}(G) \cap \{0,1\}^{m}}$, we have that $X = \diag(X)\diag(X)^{\mathrm T}$. Now, we weaken the constraint $X - \diag(X)\diag(X)^{\mathrm T} = 0$ to $X - \diag(X)\diag(X)^{\mathrm T} \succeq 0$ which is known to be equivalent to the constraints $\begin{pmatrix}
X & x \\
x^{\mathrm T} & 1
\end{pmatrix} \succeq 0$   and $\diag(X) = x$. This yields to our first SDP relaxation $SDP_{0}$ as follows.
\begin{equation}\label{SDP0}
SDP_{0}\begin{cases}
\begin{aligned}
&\text{minimize}
& & \langle Q, X \rangle\\
& \text{subject to}
& &  a_{i}^{\mathrm T}x= b_{i}, & \hspace{0.5cm} & \forall \: i \in V \backslash \{t\}\\
&
& &  \diag(X) = x, & \hspace{0.5cm} &\\
&
& &  \begin{pmatrix}
X & x \\
x^{\mathrm T} & 1
\end{pmatrix} \succeq 0. & \hspace{0.5cm} &\\
\end{aligned}
\end{cases}
\end{equation}

Here $\langle \cdot , \cdot \rangle$ denotes the trace inner product.
We show how to strengthen $SDP_{0}$ by introducing the so-called squared linear constraints.
As its name suggests, the additional constraints come from the products of the linear constraints.
Consider two linear constraints $a_{i}^{\mathrm T}x = b_{i}$ and $a_{j}^{\mathrm T}x = b_{j}$ associated with the vertices $i,j \in V \backslash \{t\}$,
the product of these two constraints is
$b_{i}  b_{j} = (a_{i}^{\mathrm T}x)  (x^{\mathrm T}a_{j}) = \langle a_{i}a_{j}^{\mathrm T}, xx^{\mathrm T} \rangle = \langle a_{j}a_{i}^{\mathrm T}, xx^{\mathrm T} \rangle$.
 Thus $\langle a_{i}a_{j}^{\mathrm T}, X \rangle = b_{i}  b_{j}$ is a valid constraint for the program $\eqref{SDP0}$.

The following result shows two properties of the squared linear constraints.
\begin{lemma}
Let $(X,x)$ satisfies
$\begin{pmatrix}
	X & x \\
	x^{\mathrm T} & 1
	\end{pmatrix} \succeq 0$, $\diag(X) = x$, and ${\langle a_{i}a_{i}^{\mathrm T}, X \rangle = b_{i}^{2}}$ for $i \in V \backslash \{t\}$.
Then

	\begin{enumerate}[topsep=0pt,itemsep=-1ex,partopsep=1ex,parsep=2ex,label=(\roman*)]
		\item the constraint $a_{i}^{\mathrm T}x = b_{i}$ is redundant for every $i \in V\backslash \{s,t\}$;
		\item the constraint $\langle a_{i}a_{j}^{\mathrm T}, X\rangle = b_{i} b_{j}$ is redundant for $i,j \in V \backslash \{t\}$ and $i \neq j$.
	\end{enumerate}
\end{lemma}
\begin{proof}
$(i)$ From the assumption, we have that $Z = X - xx^{\mathrm T} \succeq 0$ and $0 \leq x \leq 1$. From the squared linear constraint $\langle a_{i}a_{i}^{\mathrm T}, X\rangle = 0$ for $i \in V \backslash \{s,t\}$, we have
	\begin{align*}
	0 = \langle a_{i}a_{i}^{\mathrm T}, X\rangle = \langle a_{i}a_{i}^{\mathrm T}, Z + xx^{\mathrm T}\rangle  = \langle a_{i}a_{i}^{\mathrm T}, Z \rangle + \langle a_{i}a_{i}^{\mathrm T}, xx^{\mathrm T}\rangle.
	\end{align*}
	Thus,  $\langle a_{i}a_{i}^{\mathrm T}, xx^{\mathrm T}\rangle = - \langle a_{i}a_{i}^{\mathrm T}, Z \rangle \leq 0$ as $Z \succeq 0$.
However,  we also have $\langle a_{i}a_{i}^{\mathrm T}, xx^{\mathrm T}\rangle  = (a_{i}^{\mathrm T}x)^{2} \geq 0$. This implies that $a_{i}^{\mathrm T}x = 0$ for every $i \in V \backslash \{s,t\}$.
	
$(ii)$ Without loss of generality, we assume $i \neq s$ and thus $b_{i} = 0$.
As $X \succeq 0$ and $\langle a_{i}a_{i}^{\mathrm T}, X\rangle = 0$ from the assumption, it holds that $Xa_{i} = 0$ and thus $\langle a_{i}a_{j}^{\mathrm T}, X\rangle = 0$ is satisfied.
\end{proof}

The above lemma motivates us to construct the following SDP relaxation for the quadratic shortest path problem.
\begin{equation}\label{SDP_L}
(SDP_{L})\begin{cases}
\begin{aligned}
&\text{minimize}
& & \langle Q, X \rangle\\
& \text{subject to}
& &  a_{s}^{\mathrm T}x = b_{s}, & \hspace{0.5cm} &\\
&
& &  \diag(X) = x, & \hspace{0.5cm} &\\
&
& &  \begin{pmatrix}
X & x \\
x^{\mathrm T} & 1
\end{pmatrix} \succeq 0, & \hspace{0.5cm} &\\
&
& &  \langle a_{i}a_{i}^{\mathrm T}, X\rangle = b_{i}^2, & \hspace{0.5cm} & \forall \: i \in V \backslash \{t\}.\\
\end{aligned}
\end{cases}
\end{equation}

We can further strengthen $SDP_{L}$ by adding the non-negativity constraints $X \geq 0$.
This leads us to the following SDP relaxation:
 \begin{equation}\label{SDP_NL}
(SDP_{NL})\begin{cases}
\begin{aligned}
&\text{minimize}
& & \langle Q, X \rangle\\
& \text{subject to}
& &  a_{s}^{\mathrm T}x = b_{s}, & \hspace{0.5cm} &\\
&
& &  \diag(X) = x, & \hspace{0.5cm} &\\
&
& &  \begin{pmatrix}
X & x \\
x^{\mathrm T} & 1
\end{pmatrix} \succeq 0, & \hspace{0.5cm} &\\
&
& &  \langle a_{i}a_{i}^{\mathrm T}, X\rangle = b_{i}^2, & \hspace{0.5cm} & \forall \: i \in V\backslash\{t\},\\
&
& & X \geq 0 .& \hspace{0.5cm} & \\
\end{aligned}
\end{cases}
\end{equation}

Recall that $a_{t}^{\mathrm T}x = b_{t}$ is also a valid, redundant constraint for the polytope $P_{st}(G)$.
A natural question is whether the squared linear constraints induced by some redundant constraint,
e.g., $a_{t}^{\mathrm T}x = b_{t}$ tighten our relaxation? Also, may  constraints  of type
$\langle a_{i}a_{t}^{\mathrm T}, X\rangle = b_{i}  b_{t}$ ($i \in V$) further tighten  $SDP_{NL}$? The next result shows that the answer is negative.


\begin{lemma}\label{SLC_redundant}
Let $\bar{a}^{\mathrm T}x = \bar{b}$ be a redundant constraint for the path polyhedron $\eqref{stPathPolytope}$
where $\bar{a}=\sum_{i \neq t}y_{i}a_{i}$ and $\bar{b} = y^{\mathrm T}b = y_{s}$ for some $y \in \R^{n-1}$.
 Then, the  squared linear constraints
 $${\langle \bar{a}\bar{a}^{\mathrm T}, X\rangle = \bar{b}^{2}},  \text{ and }   \hspace{0.1cm} \langle a_{i}\bar{a}^{\mathrm T}, X\rangle
 =  b_{i} \bar{b}  \hspace{0.1cm} \text{ for }  \hspace{0.1cm} i \in V\backslash\{t\}$$ are redundant in the SDP relaxation $(\ref{SDP_L})$.
\end{lemma}
\begin{proof} By direct verification.
\end{proof}

It is not difficult to verify that  \eqref{SDP_L} and \eqref{SDP_NL} do not not satisfy the Slater constraint qualification.
Therefore, we derive in the following section the  Slater feasible versions of the relaxations.

\section{The Slater feasible versions of the SDP relaxations} \label{SDP-Slater}

In this section, we provide the Slater feasible versions of the SDP relaxations \eqref{SDP_L} and \eqref{SDP_NL}.
In Section \ref{slater_directedgridgraphs},  we  derive an explicit expression for the projection matrix corresponding to
the grid graph (resp.~flow grid graph) described in Example \ref{def:GridGraphs} (resp.~Example \ref{def:GridGraphs2}),

The following lemma shows that the Slater constraint qualification does not hold for the SDP relaxation (\ref{SDP_L}).
\begin{lemma}\label{Slater:nullspaceofY}
Let  $Y = \begin{pmatrix}
X & x \\
x^{\mathrm T} & 1
\end{pmatrix}$
be a feasible solution of the SDP relaxation $SDP_{L}$.
Then
\[
\mathrm{span}\{(
a_{i}^{\mathrm T}, -b_{i})^{\mathrm T} \;|\; i \in V \backslash \{t\} \}  \subseteq  \mathrm{Null}(Y).
\]
\end{lemma}
\begin{proof}
Take $\begin{pmatrix} a_{i} \\ -b_{i} \end{pmatrix}$ for $i \neq t$,
and note that the squared linear constraint
${\langle a_{i}a_{i}^{\mathrm T}, X\rangle = b_{i}^2}$  in $SDP_L$ can be written as
$$\langle \begin{pmatrix}
	a_{i} \\ -b_{i}
	\end{pmatrix}\begin{pmatrix}
	a_{i}^{\mathrm T} & -b_{i}
	\end{pmatrix}, \begin{pmatrix}
	X & x \\
	x^{\mathrm T} & 1
	\end{pmatrix}\rangle = 0.$$
	 As $Y \succeq 0$, we have
	$$\begin{pmatrix}
	X & x \\
	x^{\mathrm T} & 1
	\end{pmatrix}\begin{pmatrix}
	a_{i} \\ -b_{i}
	\end{pmatrix} = 0.$$
This shows $(a_{i}^{\mathrm T}, -b_{i})^{\mathrm{ T}} \in \mathrm{Null}(Y)$ for  $i\in V  \backslash \{t\}$.
	
\end{proof}


Define the following matrix formed by the vectors $(a_{i}^{\mathrm T}, -b_{i})^{\mathrm T}$:
\begin{equation}\label{slater:T}
T = \begin{pmatrix}
a_{1} & \cdots & a_{n-1} \\
-b_{1} & \cdots & -b_{n-1} \\
\end{pmatrix} \in \R^{m+1,n-1}.
\end{equation}

Note that the rank of $T$ is $n-1$.
It follows from \cite{tuncel,DrusWolko:17} that the minimal face that contains the feasible set of the SDP relaxation  $SDP_L$ is exposed by $TT^{\mathrm T}$.
 Assume $W \in \R^{m+1,m-n+2}$ is a matrix whose columns form a basis of the orthogonal complement to $T$, i.e., $W^{\mathrm T}T = 0$.
 Then,  we have that $Y = WUW^{\mathrm T}$ for some positive definite  $U \in \mathcal{S}^{m-n+2}$.
 This implies that substituting $Y = WUW^{\mathrm T}$ into \eqref{SDP_NL} yields a Slater feasible SDP relaxation for the QSPP.


In the sequel, we prove  that the following Slater feasible SDP relaxation is equivalent to $SDP_{L}$, see (\ref{SDP_L}).
\begin{equation} \label{SDPQSPP_L2Slater}
(SDP_{LS})\begin{cases}
\begin{aligned}
&\text{minimize}
& & \langle W^{\mathrm T}\hat{Q}W, U \rangle\\
& \text{subject to}
& &  \text{diag}(WUW^{\mathrm T}) = WUW^{\mathrm T}e_{m+1}, & \hspace{0.5cm} &\\
&
& &  e_{m+1}^{\mathrm T}WUW^{\mathrm T}e_{m+1} =1, & \hspace{0.5cm} &\\
&
& &  U \succeq 0. & \hspace{0.5cm} &
\end{aligned}
\end{cases}
\end{equation}
Here, $e_{m+1}$ is the last column of the $(m+1) \times (m + 1)$ identity matrix, and
$\hat{Q}= \begin{pmatrix} Q & 0\\0 & 0 \end{pmatrix} \in \mathcal{S}^{m+1}$.

\begin{proposition}
The SDP relaxations $SDP_{L}$ and $SDP_{LS}$ are equivalent.
\end{proposition}
\begin{proof}
Let $U$ be a feasible solution for $(\ref{SDPQSPP_L2Slater})$.
We show that $Y = WUW^{\mathrm T}$ is feasible for $\eqref{SDP_L}$.
Let $X:= Y_{1:m,1:m}$, i.e., $X$ is the  leading principal submatrix  of order $m$ of $Y$,   and $x:=\diag(X)$.
To show that  $a_{s}^{\mathrm T}x = b_{s}$, we exploit
$WUW^{\mathrm T}
\begin{pmatrix}
a_{s} \\
-b_{s}
\end{pmatrix} = 0$, from where it follows the equality.

The last set of constraints in  $(\ref{SDP_L})$ are also satisfied as
$$\langle a_{i}a_{i}^{\mathrm T}, X\rangle - b_{i}^2 = \langle \begin{pmatrix}
a_{i} \\ -b_{i}
\end{pmatrix}\begin{pmatrix}
a_{i}^{\mathrm T} & -b_{i}
\end{pmatrix}, Y\rangle = \langle W^{\mathrm T}\begin{pmatrix}
a_{i} \\ -b_{i}
\end{pmatrix}\begin{pmatrix}
a_{i}^{\mathrm T} & -b_{i}
\end{pmatrix}W, U\rangle = 0, \quad i \neq t.$$

The converse direction follows from the fact that for every feasible $Y$ in $(\ref{SDP_L})$, there exists a matrix $U \succeq 0$ such that $Y = WUW^{\mathrm T}$.
It is also easy to see that the two objectives coincide.
\end{proof}

If we add constraints $e_{i}^{\mathrm T}WUW^{\mathrm T}e_{j} \geq 0$ for every $i ,j \in \{1,\ldots,m\}$
to $SDP_{LS}$, then we obtain  the following SDP relaxation that is equivalent to $SDP_{NL}$:
\begin{equation} \label{SDP_NLS}
(SDP_{NLS})\begin{cases}
\begin{aligned}
&\text{minimize}
& & \langle W^{\mathrm T}\hat{Q} W, U \rangle\\
& \text{subject to}
& &  \text{diag}(WUW^{\mathrm T}) = WUW^{\mathrm T}e_{m+1}, & \hspace{0.5cm} &\\
&
& &  e_{m+1}^{\mathrm T}WUW^{\mathrm T}e_{m+1} =1, & \hspace{0.5cm} &\\
&
& &  WUW^{\mathrm T} \geq 0, & \hspace{0.5cm} &\\
&
& &  U \succeq 0. & \hspace{0.5cm} &
\end{aligned}
\end{cases}
\end{equation}
In the next section, we give  explicit descriptions of the projection matrices corresponding to two different types of grid graphs.

\subsection{Explicit expressions for the projection matrices} \label{slater_directedgridgraphs}

A basis of the orthogonal complement to $T$ from \eqref{slater:T}, can be obtained numerically.
However, it is computationally more efficient  to use an explicit and sparse expression for the basis $W$.
In this section, we construct $W$ for the (flow) grid graphs.

If ${C = (v_{1},\ldots,v_{k})}$ is an ordered set of vertices such that $v_{1} = v_{k}$ and each pair of vertices $\{v_{i},v_{i+1}\}$  for $i = 1,\ldots ,k-1$  are adjacent,
then $C$ is called a cycle. It is a well-known result that the null space of the incidence matrix can be identified by the vectors corresponding to the cycles in the graph.

\begin{lemma}\label{cyclenulltheorem}$\cite{biggs1993algebraic}$
Every cycle in a digraph induces a vector in the null space of the incidence matrix.
\end{lemma}
\begin{proof}
Let ${C = (v_{1},\ldots,v_{k})}$ be a cycle in the graph $G$ with $m$ arcs.
Since $v_{i},v_{i+1}$ are adjacent, then either $(v_{i},v_{i+1}) \in A$ or $(v_{i+1},v_{i}) \in A$.
We choose one of the two possible  cycle-orientations, say from $v_i$ to $v_{i+1}$, $i = 1,\ldots ,k-1$.
Define the vector $w \in \R^{m}$ such that
$$w_{e} = \begin{cases}
1 & \text{ if } e \in C  \text{ has the same orientation as } C, \\
-1 & \text{ if } e \in C  \text{ has the reverse orientation in } C, \\
0 & \text{ if } e \text{ is not in the cycle.}
\end{cases}$$
Now, for the $i$th row of the incidence matrix $a_{i}$ it follows that $a_{i}^{\mathrm T}w = 0$ for every $i \in V$.
Thus $w$ is in the null space of the incidence matrix.
\end{proof}

\textbf{The grid graphs.} We are now ready to construct vectors in the orthogonal complement of $T$ for the grid graph $G_{pq}$, see Example \ref{def:GridGraphs}.
Define cycles $(v_{i,j},v_{i,j+1},v_{i+1,j+1},v_{i+1,j})$ for $i = 1,\ldots,p-1$ and $j = 1,\ldots, q-1$,
and take vectors $w_{ij}\in \R^{m} $ as in Lemma \ref{cyclenulltheorem}.
Additionally, let $w$ be the characteristic vector of the path $(v_{1,1},\ldots,v_{1,q},\ldots,v_{p,q})$.
It is not difficult to verify the following:
\[
(a_{k}^{\mathrm T}, -b_{k})\begin{pmatrix}
w_{ij}\\ 0
\end{pmatrix} = a_{k}^{\mathrm{T}}w_{ij} = 0 \text{ and } (a_{k}^{\mathrm T}, -b_{k})\begin{pmatrix}
w \\ 1
\end{pmatrix} = a_{k}^{\mathrm{T}}u - b_{k} = 0,
\]
for  $i = 1,\ldots,p-1$ and $j = 1,\ldots, q-1$ and $k \in V$.
Thus, the following ${m-n+2}$ independent vectors
$$\begin{pmatrix}
w\\ 1
\end{pmatrix} \cup
\left \{ \begin{pmatrix}
w_{ij}\\ 0
\end{pmatrix} \;|\; i=1,\ldots,p-1, j = 1,\ldots, q-1 \right \}$$
span the null space of the column space of $T$.
Thus, we have
$${W =\begin{bmatrix}
	w & w_{1,1} & \ldots & w_{p-1,q-1} \\
	1 & 0 & \ldots  & 0
	\end{bmatrix}\in \R^{m+1,m-n+2}}.$$

\textbf{The flow grid graphs.}
 Here,  we construct vectors in the orthogonal complement of $T$ for the flow grid graph with $pq+2$ vertices, see Example  \ref{def:GridGraphs2}.
We first define cycles $(v_{i,j},v_{i,j+1},v_{i+1,j+1},v_{i+1,j})$ for $i = 1,\ldots,p-1$ and $j = 1,\ldots, q-1$, and cycles $t_{s,i} = (s,v_{i,1},v_{i+1,1})$,
$t_{i,t} = (t,v_{i,q},v_{i+1,q})$ for $i = 1,\ldots,p-1$. Then, we take vectors $w_{ij}\in \R^{m}$ as in Lemma \ref{cyclenulltheorem} for the defined cycles.
Let $w \in \R^{m}$ be the characteristic vector of the path $(s,v_{1,1},\ldots,v_{1,q},t)$. Similar to the construction of $W$ for the grid graphs,
we obtain an explicit expression for $W \in \R^{m+1,m-n+2}$ from vectors $w_{ij}$ and $w$.

\section{SDP relaxations and directed acyclic graphs} \label{sect:Special graphs}

Most of the constraints in the SDP relaxations $SDP_{L}$ and $SDP_{NL}$ are derived from the incidence matrix of the underlying graph.
Therefore,   constraints in the relaxations  differ for different graphs.
In this section we show some additional properties of  the feasible sets of $SDP_{L}$ and $SDP_{NL}$
when the considered graph is acyclic.

We show first results for graphs in which  every $s$-$t$ path has the same length.

\begin{lemma} \label{gridgraph_lemma1}
Let $G_{p,q}$ be the grid graph, and
$Y=\begin{pmatrix}
X & x \\
x^{\mathrm T} & 1
\end{pmatrix}$  feasible for $SDP_{L}$. Then,
\begin{enumerate} 
	\item[(i)] $Xe = L x$;
	\item[(ii)] $x^{\mathrm T}Xe = L^{2}$ and $e^{\mathrm T}Ye = (L+1)^{2}$,
\end{enumerate}
where $L = p + q -2$ is the length of the $s$-$t$ path.
\end{lemma}
\begin{proof}
Let $T\in \R^{m+1,n-1}$ be the matrix defined in \eqref{slater:T}. Note that the columns of $T$ can be indexed by the vertices $v_{ij}$ of $G_{p,q}$.
Define the vector $w \in \R^{n-1}$ such that the element of $w$ that corresponds to the vertex $v_{ij}$ equals $p+q-i-j$.
Then we have $Tw = (e^{\mathrm T}, -L)^{\mathrm T}$.

Since the column space of the matrix $T\in \R^{m+1,n-1}$ spans the null space of $Y$, the vector $(e^{\mathrm T}, -L)^{\mathrm T}$ is also in the null space of
$Y$, i.e., $Y(e^{\mathrm T}, -L)^{\mathrm T} = 0$. From here it follows that $Xe = L x$.
Using the fact that $e^{\mathrm T}x = L$, we can derive (ii) from (i).
\end{proof}

Clearly Lemma \ref{gridgraph_lemma1} also holds for feasible solutions of $SDP_{NL}$.
We should note that the similar proof follows for any graph in which every $s$-$t$ path has the constant length.

In the following lemma we show that a particular zero pattern holds for feasible points of $SDP_{NL}$ when the considered graph is acyclic.

\begin{lemma}\label{SDP_NLproperty2}
Let $(X,x)$ be feasible for $SDP_{NL}$.
If $G$ is a  directed acyclic graph,
then $X_{ef} = 0$ whenever there exists no $s$-$t$ path containing both arcs $e$ and $f$.
\end{lemma}

\begin{proof}
Let $(v_{1},\ldots,v_{n})$ be a topological ordering of the directed acyclic graph $G$, and $s=v_{1}$ and $t = v_{n}$.
Assume without loss of generality that $e=(v_{i},v_{j}), f=(v_{k},v_{l})$ and $i < k$.

We define a subset $S$ of vertices based on the order of $v_{j}$ and $v_{k}$.
If $j > k$, then   $S := \{v_{1},\ldots,v_{k}\}$.
If $j < k$, then we define $$S  := \{v_{1},\ldots,v_{j-1}\} \cup \{ v\in  \{v_{j+1},\ldots,v_{k-1}\} \;|\; \text{there exists a path from $v$ to $v_{k}$} \}  \cup \{v_{k}\}.$$
We claim that there does not  exist an arc from $V\backslash S$ to $S$.
The claim is trivial when $j > k$.
Therefore we discuss the case when $j < k$.
Suppose for the sake of contradiction that there exists  an arc $(v_{i'},v_{j'})$ with $v_{i'} \in V\backslash S$ and $v_{j'} \in S$.
By the construction of $S$, we know $i'$ and $j'$ satisfy $j \leq i' < j' \leq k$. As $v_{j'} \in S$ and $j < j' \leq k$,
we have that there is a path from $v_{j'}$ to $v_{k}$. Since $(v_{i'},v_{j'})$ is an arc of $G$, this means that there is also a path from $v_{i'}$ to $v_{k}$, and thus $v_{i'} \in S$.
This contradicts  the assumption $v_{i'} \in V\backslash S$ for  $i' > j$.
If $i' = j$, then this contradicts the assumption that there does not exist  $s$-$t$ path containing both arcs $e$ and $f$.

Let $A'$ be the set that  contains arcs from $S$ to $V\backslash S$. Thus $e,f \in A'$. Define $\lambda \in \R^{n-1}$ such that $\lambda_{i} = 1$ if $i \in S$, and zero otherwise.
Because there are no arcs from $V\backslash S$ to $S$, we know that $a:= \sum_{i}\lambda_{i}a_{i}$ is a vector such that $a_{e} = 1$ if $e \in A'$, and zero otherwise.
Clearly, $\lambda^{\mathrm T}b = 1$. Thus $a^{\mathrm T}x  =1$ is a valid constraint, which has the interpretation that every $s$-$t$ path contains exactly one arc in $A'$.
Applying Lemma \ref{SLC_redundant}, we know that the squared linear constraint $\langle aa^{\mathrm T} ,X \rangle = 1$ is a redundant constraint.

Let $X_{1}$ be the submatrix of $X$  associated to the arcs in $A'$.
From $a^{\mathrm T}x  =1$ and $\langle aa^{\mathrm T} ,X \rangle = 1$,
we have  $\tr(X_{1}) = 1$ and $\langle J,X_{1} \rangle = 1.$ As $X_{1} \geq 0$, it holds that $X_{1}$ is a diagonal matrix.
Thus $X_{e',f'} = (X_{1})_{e',f'} = 0$ for every distinct arcs $e',f' \in A'$. In particular, we have $X_{e,f} = 0$ as $e,f \in A'$.
\end{proof}

It is not difficult to verify that Lemma \ref{SDP_NLproperty2} does not hold for feasible points in $SDP_{L}$.
Therefore, in order to tighten the $SDP_{L}$  relaxation one may enforce  constraints $X_{ef} = 0$ for $e,f\in A$,
whenever there exists no $s$-$t$ path containing both arcs $e$ and $f$.
We denote so obtained relaxation by $SDP_{L+}$ and its Slater feasible version $SDP_{LS+}$.
Note that for a directed acyclic graph it is not difficult to determine all such pairs of arcs, but this is not the case in general.
Table \ref{table:SDP_LvsSDP_L+} shows that $SDP_{LS+}$ provides significantly better bound than $SDP_{LS}$.
Therefore, in Section \ref{sect:NumerExper} we compute $SDP_{LS+}$ for the QSPP instances on the grid graphs.

\begin{table}[H]
	\centering
		\begin{tabular}{|cc|cc|}
			\hline
			$n$ & $m$ & $sdp_{ls}$ & $sdp_{ls+}$  \\ \hline
			400  & 760  & -1057.81 & 393.38  \\
			400  & 760  & -1052.84 & 428.69  \\
			400  & 760  & 1146.86 & 3109.75  \\
			400  & 760  & 2846.78 & 4773.37  \\  \hline
	\end{tabular}
\caption{SDP bounds for the QSPP instances on $G_{20,20}$.} \label{table:SDP_LvsSDP_L+}
\end{table}

\section{The alternating direction method of multipliers} \label{section_admm}

Although semidefinite programming has proven effective for combinatorial optimization problems,
SDP solvers based on interior-point methods might have  considerable  memory demands already for medium-scale problems.
The alternating direction method of multipliers is a first-order method for convex problems developed in the 1970s.
This method decomposes an optimization problem into subproblems that may be easier to solve.
This feature makes the ADMM well suited for large-scaled problems. For state of the art in theory and applications of the ADMM, we refer the interested readers to \cite{boyd2011distributed}.
The study of the ADMM for solving semidefinite programming problems can be found in \cite{wen2010alternating, povh2006boundary,oliveira2015admm}.

Oliveira, Wolkowicz and Xu \cite{oliveira2015admm} propose solving an SDP relaxation for the quadratic assignment problem  using the ADMM.
Their computational experiments show that the proposed variant of the ADMM exhibits remarkable robustness, efficiency, and even provides improved bounds.
In this section, we briefly outline the approach from \cite{oliveira2015admm} and show how to apply it for solving our SDP relaxations of the QSPP.

We consider now the SDP relaxation $SDP_{NLS}$.
In order to  obtain a separable objective,  we replace $WUW^{\mathrm T}$ by $Y$, and  add the coupling constraint $Y = WUW^{\mathrm T}$.
Furthermore, we  add  the  redundant constraint $Y \leq 1$, which is known to improve the performance of the algorithm, see \cite{oliveira2015admm}.
This yields the following program:
\begin{equation}  \label{ADMM_SDP1}
\begin{aligned}
&\text{minimize}
& & \langle \hat{Q}, Y \rangle\\
& \text{subject to}
& &  \text{diag}(Y) = Ye_{m+1}, & \hspace{0.5cm} &\\
&
& &  Y_{m+1,m+1} =1, & \hspace{0.5cm} &\\
&
& &  Y = WUW^{\mathrm T}   , & \hspace{0.5cm} &\\
&
& &  0 \leq Y \leq 1, \; U \succeq 0. & \hspace{0.5cm} &
\end{aligned}
\end{equation}

The augmented Lagrangian of \eqref{ADMM_SDP1} corresponding to the linear
constraint $Y = WUW^{\mathrm T}$ is given by:
\begin{align*}
\mathcal{L}(U,Y,Z)  = \langle \hat{Q}, Y \rangle + \langle Z, Y - WUW^{\mathrm T} \rangle + \frac{\beta}{2} \vectornorm{Y - WUW^{\mathrm T}}^{2}_{F},
\end{align*}
where $Z \in \mathcal{S}^{m+1}$ is the dual variable, and $\beta > 0$  the penalty parameter, and $\| \cdot \|_F$ the Frobenius norm.
The alternating direction method of multipliers solves in the $(k+1)$-th iteration the following subproblems:
\begin{align}
U^{k+1} & = \arg \min_{U  \succeq 0 } \mathcal{L}(U,Y^{k},Z^{k}), \label{ADMM:U_sub}\\
Y^{k+1} & = \arg \min_{Y  \in P} \mathcal{L}(U^{k+1},Y,Z^{k}), \label{ADMM:Y_sub}\\
Z^{k+1} & = Z^{k} + \gamma \cdot \beta (Y^{k+1} - WU^{k+1}W^{\mathrm T}),\label{ADMM:Z_sub}
\end{align}
where \begin{align} \label{ADMM:P}
P = \{Y \in \mathcal{S}^{n} \; | \; \diag(Y) = Ye_{m+1},\;  Y_{m+1,m+1} =1,\; 0 \leq Y \leq 1\}.
\end{align}
Here $\gamma \in (0,\frac{1+\sqrt{5}}{2})$ is the step-size for  updating the dual variable $Z$, see e.g., \cite{wen2010alternating}.

Let $W$ be normalized such that $WW^{\mathrm T} = I$.
Then, the $U$-subproblem  reduces to the following:
\[
\begin{array}{rcl}
U^{k+1} & =& \arg \min\limits_{U  \succeq 0 }  \langle Z^{k}, Y^{k} - WUW^{\mathrm T} \rangle + \frac{\beta}{2} \vectornorm{Y^{k} - WUW^{\mathrm T}}^{2}_{F} \label{ADMM:Usub} \\[1.5ex]
& =& \mathcal{P}_{\mathcal{S}_{+}} (W^{\mathrm T}(Y^{k}  + \frac{1}{\beta}Z^{k})W), \nonumber
\end{array}
\]
where $\mathcal{P}_{\mathcal{S}_{+}}(M)$ is the projection to the cone of positive semidefinite matrices.

The closed-form solution of the $Y$-subproblem is as follows:
\begin{align} \label{ADMM:Y_sub_sol}
Y^{k+1} & =  \arg \min_{Y  \in P} \vectornorm{Y - WU^{k+1}W^{\mathrm T} + \frac{\hat{Q} +Z^{k}}{\beta}}^{2}_{F}   \\
& = \begin{cases}
\min\{1, \max\{0,\hat{Y}_{i,j}\}\}   & \text{ if } i<j < m + 1, \\
\min\{1, \max\{0,\frac{1}{3}\hat{Y}_{i,i} +\frac{2}{3}\hat{Y}_{i,m+1} \} \}   & \text{ if } i = j < m + 1, \\
\min\{1, \max\{0,\frac{1}{3}\hat{Y}_{i,i} +\frac{2}{3}\hat{Y}_{i,m+1} \} \}   & \text{ if } i < j = m+1, \\
1  & \text{ if } i = j = m + 1, \\
\end{cases} \nonumber
\end{align}
where
\begin{equation} \label{Yhat}
\hat{Y} = WU^{k+1}W^{\mathrm T} - \frac{\hat{Q}+Z^{k}}{\beta}.
\end{equation}

In a similar fashion, we can solve $SDP_{L}$ by the ADMM. Note that the non-negativity constraints are very strong cuts for the SDP relaxations.
These constraints are also extremely expensive when solving SDP relaxations with interior-point methods.
However, the complexity of the ADMM only slightly increases when the non-negativity constraints are imposed to strengthen the relaxation, as noticed in \cite{oliveira2015admm}.

\medskip

\textbf{Lower and upper bounds.} \label{ADMM_lbub}
To solve an SDP problem to a high accuracy by an ADMM-based solver can be prohibitively expensive.
Therefore Oliveira et al.~\cite{oliveira2015admm} consider solving \eqref{ADMM_SDP1} to a moderate accuracy, while obtaining a valid bound.
We implement their approach for the QSPP. This is explained in the sequel.

Let  $P$ be the feasible set for $Y$-subproblem, see \eqref{ADMM:P}, and ${\mathcal{Z} = \{Z \;|\;  W^{\mathrm T}ZW \preceq 0 \}}.$
The Lagrangian dual of \eqref{ADMM_SDP1} is as follows:
\begin{align*}
 \max_{Z} \min_{U\succeq 0,Y \in P} \langle \hat{Q}, Y \rangle + \langle Z, Y - WUW^{\mathrm T} \rangle  =  \max_{Z \in \mathcal{Z}} \min_{Y\in P}  \langle \hat{Q} + Z, Y \rangle,
\end{align*}
and satisfies weak duality.
Thus, for a feasible dual variable $Z \in \mathcal{Z}$
\begin{equation}\label{admm_upperbound}
g(Z) = \min_{Y\in P}  \langle \hat{Q} + Z, Y \rangle
\end{equation}
provides a lower bound for \eqref{ADMM_SDP1}.
Now, let $(\bar{U},\bar{Y},\bar{Z})$ be the output of the ADMM for  \eqref{ADMM_SDP1}.
The projection of $\bar{Z}$ onto $\mathcal{Z}$ gives us a feasible $Z$ that we  use to compute a lower bound.
The projection can be done efficiently, as explained in \cite{oliveira2015admm}.

One can also compute an upper bound for the problem from the output $(\bar{U},\bar{Y},\bar{Z})$ of the ADMM for \eqref{ADMM_SDP1}.
We define ${d} \in \R^{m}$ such that $d_{ii} := \bar{Y}_{ii}$ for $i=1,\ldots,m$, and solve the following linear programming problem:
\begin{align} \label{admm_lpupper}
\underset{x \in \R^{m}}{\text{minimize}} \;\; d^{\mathrm T}x \;\; \text{ s.t. } x \in P_{st}(G).
\end{align}
This gives a feasible $s$-$t$ path $x$ whose quadratic cost is an upper bound for the QSPP.
We note that the quality of the upper bound from \eqref{admm_lpupper} heavily depends on the quality of the ADMM output $\bar{Y}$.

\section{Improving  performance of the ADMM} \label{sect:FastADMM}

Oliveira et al., \cite{oliveira2015admm} (see also  Section \ref{section_admm})  show  how  to obtain a lower bound for the optimization  problem
from the output of the ADMM-based algorithm that solves an SDP relaxation  to a moderate accuracy.
So obtained bounds are weaker than the bounds obtained using higher accuracy.
Clearly,  there is a trade-off between the computational effort and the quality of the SDP  bound.
Our numerical results show that within a branch-and-bound framework it is preferable  to use slightly weaker bounds  that can be efficiently computed.

Therefore, in this section we study  how to improve the performance of the ADMM algorithm in the first few hundreds of iterations.
We restrict here   on graphs for which every $s$-$t$ path has the same length.

Let us first  recall the projection onto the simplex problem. The projection of a vector onto the simplex is a well-studied problem.
The simplex is defined as a set of non-negative vectors whose entries sum up to a non-negative number $a$:
$\Delta(a) := \{ x \in R^{n} \;|\;  x \geq 0, \; \sum_{i=1}^{n} x_{i} = a \}.$
Then, the minimization problem
$$\mathcal{P}_{a}(y) =  \arg \min_{x \in \Delta(a)} \vectornorm{x - y}$$
is a projection onto the simplex $\Delta(a)$.
We refer the reader to \cite{condat2014fast} for a comprehensive overview of this problem.
It is also known that the projection onto the simplex can be solved in $\mathcal{O}(n \log n)$, see \cite{held1974validation}.

\medskip
Suppose now that the length of every $s$-$t$ path is equal to $L$.
Then, the constraint $e^{\mathrm T}Ye = (L+1)^2$ is a valid constraint for $SDP_{NLS}$, see \eqref{ADMM_SDP1} and Lemma  \ref{gridgraph_lemma1}.
Let us show that this constraint can be incorporated in a way that our ADMM algorithm retains fast iterates.

Define matrix $S \in \mathcal{S}^{m+1}$ such that $S_{ii}=S_{i,m+1}=S_{m+1,i} = 1$ for $i = 1,\ldots,m+1$, and zero otherwise.
Then the constraints $ \langle S, Y \rangle  = 3\cdot L+1$ and ${\langle e^{\mathrm T}e - S , Y \rangle = L\cdot(L-1)}$ are valid for $SDP_{NLS}$.
Clearly,  the $U$-update \eqref{ADMM:U_sub} and $Z$-update \eqref{ADMM:Z_sub} in the ADMM for solving $SDP_{NLS}$ are not affected by adding those constraints.
The only change is in the feasible region $P$ (see \eqref{ADMM:P}) of the $Y$-subproblem \eqref{ADMM:Y_sub}.
Let us define the new feasible region $\bar{P} := P_{1} \cap P_{2}$ where
\begin{align*}
&P_{1} =  \{Y \in \mathcal{S}^{n} \; | \; \langle S, Y \rangle  = 3 L+1, \;  \; Y \geq 0, \;  \;  \diag(Y) = Ye_{m+1},\;  Y_{m+1,m+1} =1\},  \\
&P_{2} =  \{Y\in \mathcal{S}^{n} \; | \; \langle e^{\mathrm T}e - S, Y \rangle = L (L-1),\;  Y \geq 0 \} .
\end{align*}
In the sequel we show that the new $Y$-subproblem can be solved efficiently.
This is accomplished by splitting the problem into two subproblems based on the nonzero entries in $S$ and $e^{\mathrm T}e - S$ as follows:
\begin{align*}
\min_{Y  \in \bar{P}} \mathcal{L}(U^{k+1},Y,Z^{k}) & =   \min_{Y  \in \bar{P}} \vectornorm{Y - \hat{Y}}^{2}_{F}
=   \min_{Y  \in P_{1}} \vectornorm{Y - \hat{Y}}^{2}_{F} + \min_{Y  \in P_{2}} \vectornorm{Y - \hat{Y}}^{2}_{F},
\end{align*}
where $\hat{Y}$ is given in \eqref{Yhat}.
Each of the two  minimization problems on the right-hand side above is the projection onto the simplex problem.

For the first problem, we have that  ${\min_{Y  \in P_{1}} \vectornorm{Y - \hat{Y}}^{2}_{F} = \min_{Y  \in P_{1}} \sum_{i=1}^{m} (Y_{ii}-\hat{y}_{i})^2}$, where  $\hat{y} \in \R^{m}$ is a vector such that
$\hat{y}_{i} = \frac{1}{3}\hat{Y}_{ii}+\frac{2}{3}\hat{Y}_{i,m+1}$ for $i = 1,\ldots,m$.
Then, the minimizer of the first problem can be found via the following projection onto the simplex
${\mathcal{P}_{L}(\hat{y}) = \arg \min_{x \in \Delta(L)} \vectornorm{x - \hat{y} }.}$
More precisely, the explicit solution of the first problem is given by
\[
Y_{m+1,m+1} = 1 \text{ and } Y_{ii} = Y_{i,m+1} = Y_{m+1,i} = (\mathcal{P}_{L}(\hat{y}))_{ii} \quad\text{ for }i = 1,\ldots,m.
\]

For the second problem, we take the vector $\hat{y}\in \R^{{m \choose 2}}$ whose entries are indexed by the nonnegative entries
$(i,j)$, $i < j$, in $e^{\mathrm T}e - S$  such that $\hat{y}_{ij} = \hat{Y}_{ij}$.
Then, the second problem is equivalent to the projection onto the simplex  $\mathcal{P}_{L^{2}-L}(\hat{y})$, and
the solution is given by $$Y_{ij} = Y_{ji} = (\mathcal{P}_{L^{2}-L}(\hat{y}))_{ij} \quad\text{ for } i<j < m +1.$$

To sum up, we add redundant constraints to $SDP_{NLS}$ and obtain a different $Y$-subproblem from \eqref{ADMM:Y_sub}.
 The new  $Y$-subproblem  can be decomposed into two projections onto the simplex, which can be  solved efficiently.



\begin{figure}[h]
	\centering
	\subfloat[bounds for an instance on $G_{20,20}$]{{\includegraphics[width=7cm]{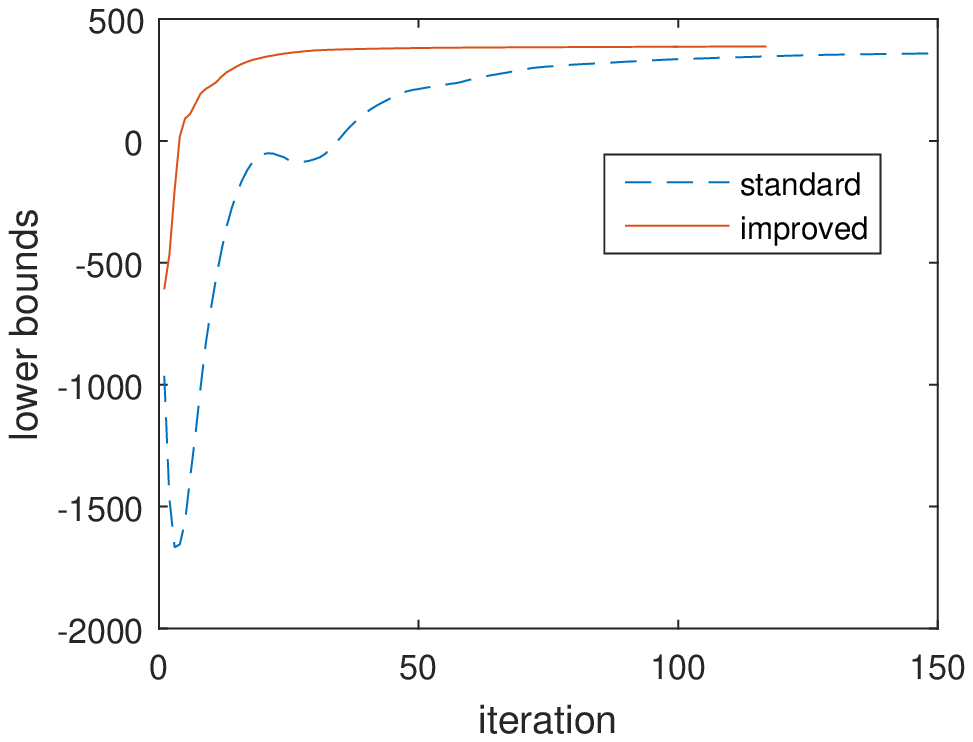} }\label{constantlength}}
	\quad
	\subfloat[bounds for an instance on  $G_8$]{{\includegraphics[width=7cm]{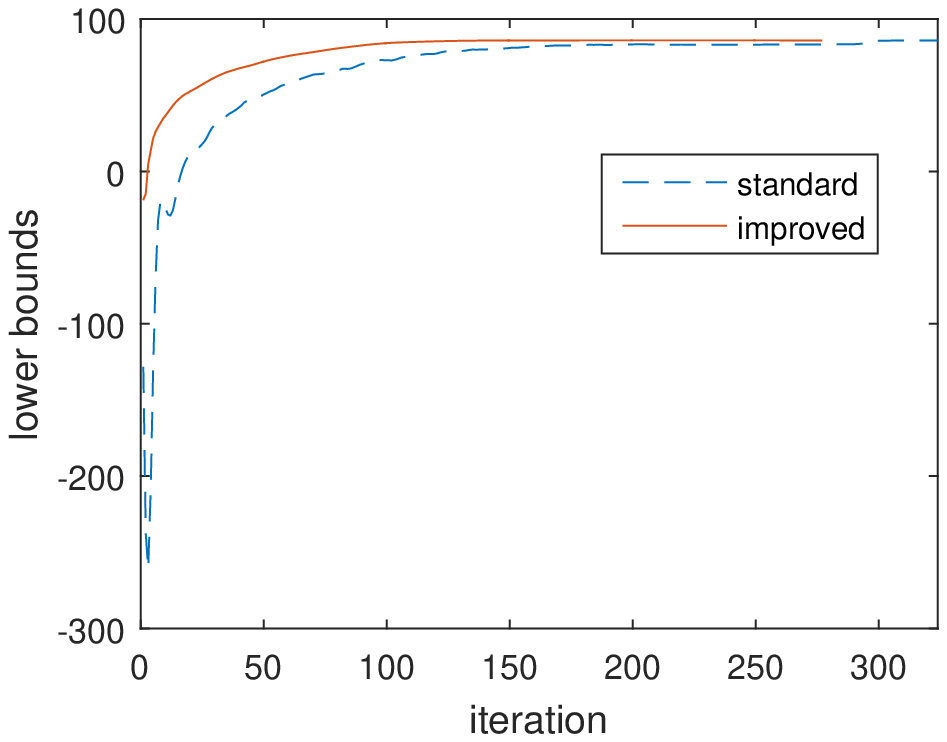}}\label{constantlength2}}%
	\caption{Lower bounds for the QSPP }%
	\label{fig:example}%
\end{figure}

We test an impact of adding $e^{\mathrm T}Ye = (L+1)^2$ to $SDP_{NLS}$ on  performance of the ADMM algorithm.
Figure \ref{constantlength} (resp.~\ref{constantlength2}) presents lower bounds computed in the first few hundred iterations of the ADMM algorithm
for  a QSPP instance on $G_{20,20}$ (resp.~$G_8$).
The the dashed lines present bounds obtained without using the projections onto the simplex, while the solid line presents bounds obtained by using the projections.
We observe that the bounds obtained by using the redundant constraints  are better.
The  lines end up at the points in which the stopping criteria is satisfied, see the next section for details.
Clearly, one should  incorporate additional redundant constraints in order to obtain a better performance of the algorithm in the earlier iterates.
Since the dashed line stabilizes after the initial fluctuations, the effect of the redundant constraints in not beneficial in the long run.

\subsection{A branch-and-bound algorithm} \label{sect:BnB}

We describe here  our branch-and-bound algorithm  for solving the QSPP on the grid graphs.
The B$\&$B algorithm combines our strongest SDP relaxation $SDP_{NLS}$, the ADMM-based solver, simulated annealing heuristics,
and \eqref{admm_lpupper} to solve instances of the grid graph.

Our branching rule is as follows: staring with the vertex $i$, we branch over each of its unvisited neighbors $j$, i.e., $e = (i,j) \in A$.
If we branch over an arc $e$, then the linear cost of each arc $f$ is increased by $2 q_{e,f}$.
The linear cost of each of the outgoing arcs from  vertex $j$ is increased by $q_{e,e}$.
This leads to two smaller quadratic shortest path problems, and each subproblem partitions the original QSPP.

The bounding scheme uses the semidefinite programming relaxation $SDP_{NLS}$ with redundant constraints in the way as described in this section.
At each node of the branching tree, we compute a lower bound for the current node using the ADMM algorithm,
and also update the best upper bound found so far.
At the root node, we compute an upper bound by using our simulated annealing algorithm.
In all other nodes we solve  the linear programming problem \eqref{admm_lpupper} in order to get an upper bound.

The settings of  the ADMM turn out to be crucial for the performance of the branch-and-bound algorithm. The ADMM is notorious for its slow convergence to high accuracy.
Therefore we compromise this by using the SDP relaxation with additional redundant constraints, and low-precision in the way as described in this section.
Here, we set the stopping criteria as follows: if the primal and dual residual is less than $0.5$, and the difference
between the objective values of two consecutive iterations is less than $0.1$ for at least $15$ iterations in a row, then we terminate the algorithm.
This termination rule still yields  lower bounds comparable to those obtained with high precision tolerance.
However, the computational cost is lower.

An implementation details of the branch-and-bound algorithm that incorporates SDP bounds and the ADMM for solving
the quadratic assignment problem  can be found in the master thesis of Liao \cite{liao2016QAPbbadmm}.

\section{Numerical experiments} \label{sect:NumerExper}

In this section we present numerical results for the quadratic shortest path problem.
We compute $SDP_{LS+}$ and  $SDP_{NLS}$ bounds by using the ADMM.
For comparison reasons we also compute lower bounds from \cite{rostami2016quadratic}.
We  present numerical results for solving  to optimality the QSPP on the grid graphs
by using our B$\&$B algorithm as described in see Section \ref{sect:FastADMM}.

The experiments are implemented in Matlab on the machine with an Intel(R) Core(TM) i7-6700 CPU, 3.40GHz and 16 GB memory.
 The bounds from  \cite{rostami2016quadratic}  are solved by Cplex \cite{cplex20007} and the Bellman-Ford algorithm.

To test and compare various bounding techniques for the QSPP,
we use different types of instances. First of all, we define the random variable
$W(d)$ for  fixed $d \in (0,1]$ such that $\mathcal{P}(W(d) = 0) = 1-d$ and $\mathcal{P}(W(d) = i) = d/10$ for $i \in \{1,\ldots,10\}$.
Now we present the instances as follows.
\begin{enumerate} 

\item[(i)] {\sc grid}1  is a QSPP instance on the grid graph from Example \ref{def:GridGraphs}. The cost $q_{e,f} = q_{f,e} = w_{ef}(d)$
is the realization of the random variable $W_{ef}(d)$ for $d \in (0,1]$, for each pair of distinct arcs $e$ and $f$.
Similarly, we take the linear costs $q_{e,e} = w_{e}(d)$ for each arc $e$.

\item[(ii)] {\sc grid}2  is a QSPP instance on the flow grid graph defined in Example \ref{def:GridGraphs2}. The costs are produced in the same way as for {\sc grid}1.
We note that both {\sc grid}1 and {\sc grid}2 are used in $\cite{buchheim2015quadratic,rostami2016quadratic}$.

\item[(iii)] {\sc grid}3  is a QSPP instance on the flow grid graph.
The difference between {\sc grid}2 and {\sc grid}3  is that {\sc grid}3 depends on two  parameters;
 $d$ and $d'$  that  are related to the horizontal and vertical arcs, respectively.
In particular, we set the quadratic cost $q_{e,f} =q_{f,e} = w_{ef}(d)$ if $e$ and $f$ are horizontal arcs, and $q_{e,f}=q_{f,e}= w_{ef}(d')$ if $e$ or $f$ are vertical arcs.
Similarly, we set the linear cost $q_{e,e} = w_{e}(d)$ if $e$ is a horizontal arcs, and $q_{e,e} = w_{e}(d')$ if $e$ is a vertical arcs.

\item[(iv)] {\sc grid}4  is a QSPP instance on the double-directed grid graph, see Example \ref{def:GridGraphs3}.
For the case that  the arcs $e$ and $f$ are of the form $(v_{i,j},v_{i+1,j})$ or $(v_{i,j},v_{i,j+1})$,
we set the linear costs $q_{e,e} = w_{e}(d)$, $q_{f,f} = w_{f}(d)$, and interaction costs $q_{e,f} =q_{f,e} = w_{ef}(d)$. Here $d \in (0,1]$.
 All other costs are zero.

\item[(v)] {\sc par}-{\sc k}  is a QSPP instance  on the incomplete $K$-partite graph, see Example \ref{def:Kpartite}.
 We set $V_1=\{s\}$, $V_2=\{t\}$, and $|V_{i}| = K$ for $i=2, \ldots, K-1$. Thus,  we  have that $|V| = K(K-2) + 2$ and $|A| =  K^2(K-3) + 2K$.
 The quadratic and linear costs of the arcs in $G_{K}$ are generated  in the same way as for {\sc grid}1 instances.

\end{enumerate}

The size of the grid graph $G_{p,q}$ depends on the parameters $p$ and $q$. If $p=q$, then we say that the associated graph is a {\sc square}  grid graph.
Similarly, a grid graph with $4p=q$ is called a {\sc long} grid graph, and $p = 4q$ is called a {\sc wide} grid graph.
These test graphs are introduced in \cite{kovacs2015minimum}, and used in \cite{rostami2016quadratic}.

All the SDP  bounds are solved approximately by our ADMM-based algorithm, see Section \ref{section_admm}.
The ADMM stops when either the maximum number of iterations $25000$ is reached, or when the tolerance $1e-5$ is reached.
We heuristically take $\gamma = 1.618 $ and $\beta = \sqrt{n}/2$.
We note that smaller tolerance significantly increases running time of the algorithm,  but yields small improvement in the value of the bound.
To solve the QSPP by using the B$\&$B algorithm, we use different tolerance and $SDP_{NLS}$ with additional constraints as described in Section \ref{sect:BnB}. \\

Since we compare our bounds with several bounding approaches from the literature, we briefly outline those.
Rostami et al. \cite{rostami2016quadratic} proposed a reformulation scheme by constructing an equivalent QSPP such that the linear cost has more impact on the solution value.
The procedure can be applied iteratively to obtain increasingly better lower bounds.
We test here this iterative approach.
Our results show that it is the most efficient to stop the iterative procedure when the improvement between the $(k-1)$th and the $k$th iteration is less than $\min\{k, 10\}$ percentage.
This results with the best trade-off between the computed bound and its computational  cost.
The obtained lower bound is denoted here by $RBB$.
We also compute the Gilmore-Lower type bound ($GL$) for the QSPP, see \cite{rostami2016quadratic}.
Finally, we note that $RBB$ at the first iteration equals the $GL$.\\

\noindent
\textbf{Test results.}

In what follows we present and summarize numerical results.
\begin{enumerate}
\item[(i)] We report our results in Table \ref{table:LBdirectedGridgraphs} and  \ref{table:LBdirectedGridgraphsTime} for the {\sc grid}1  instances on
{\sc square} grid graphs. The size of the instances ranges from $220$ to $760$ arcs. For each size, we generate four instances with $d = 0.2 ,0.2, 0.8, 0.8$.

Table \ref{table:LBdirectedGridgraphs} reads as follows. In the first two columns, we list the number of the vertices and the number of arcs in the grid graph $G_{p,q}$, respectively.
In particular, we have $p = q = \sqrt{n}$, and $m = 2pq-p-q$. The third and fourth columns list the Gilmore-Lower bounds and the reformulation-based lower bounds $RBB$, respectively.
The fifth column provides   the lower bound $SDP_{LS+}$. Note that $SDP_{LS+}$ stands for the SDP bound $SDP_{LS}$ with the additional zero pattern, see Section \ref{sect:Special graphs}.
The sixth column provides the lower bounds $SDP_{NLS}$, and the seventh column contains the associated upper bounds.
Here, the upper bound is obtained by solving \eqref{admm_lpupper} where $d$ is derived from the output of the ADMM for $SDP_{NLS}$.
The eighth column presents the lower bound of the root node with the tolerance $0.5$ in the branch-and-bound tree (see also Section \ref{sect:BnB}),
and the last column is the optimal value computed by our branch-and-bound algorithm.
Table \ref{table:LBdirectedGridgraphsTime} presents the computational times and the number of iterations required to obtain bounds in Table \ref{table:LBdirectedGridgraphs}.
The column marked with ($s$) is the running time in seconds, ($it$) the number of iterations, and ($n$) the number of vertices in the branching tree.
This labeling also applies to the other tables.

We observed that both, the GL bounds and the RBB bounds heavily depend on the choice of the parameter $d$.
If $d$ is small, the bound is rather weak. For larger $d$, $RBB$ is usually $50$\% to $80$\% of the optimal value.
It is worth to note that  RBB is a linear programming-based bound.

$SDP_{LS+}$, provides significantly better bounds than those obtained from the reformulation scheme.
However,   $SDP_{NLS}$ yields extremely strong lower bounds.
For almost all of the tested instances with $n \leq 225$, $SDP_{NLS}$ provides tight bounds  in a short time.
Note also that in most of the cases  the $SDP_{NLS}$ bounds are computed faster than the  $SDP_{LS+}$ bounds.
This is due to the fact that the ADMM-based  algorithm requires more iterations to reach the tolerance for a weaker relaxation than for a stronger relaxation.
The upper bounding procedure from Section \ref{ADMM_lbub} yields a good upper bound only when  $SDP_{NLS}$ is close to the optimal value.

Our B$\&$B algorithm is able to solve to optimality instances with $760$ arcs  within $3$ minutes (!).
Also, our branch-and-bound algorithm  solves instances on the $25\times 25$ grid graph (1200 arcs) within $30$ minutes, and
instances on the $26 \times 26$  grid graph (1300 arcs)  within $50$ minutes.

We note  that the interior-point algorithm from Mosek \cite{mosek2010mosek}
solves the SDP  relaxation $SDP_{NLS}$ for an instance with 480 arcs  in $45$ minutes.
Cplex solver cplexqp is  capable to handle the QSPP instances  with  $m \leq 364$ arcs within one hour.

\item[(ii)]
In Table \ref{table:GRID2SQUAREvalue}, \ref{table:GRID2LONGvalue} and \ref{table:GRID2WIDEvalue} we report the results for the  {\sc grid}2  instances.
Those tables read similarly to the Tables \ref{table:LBdirectedGridgraphs} and \ref{table:LBdirectedGridgraphsTime}.
For each different size of  $m$, we generate four {\sc grid}2 instances with $d = 0.2,0.2,0.8,0.8$, respectively.
It turns out that {\sc grid}2 instances are easy instances.
In particular,  the computed SDP bounds presented in Tables  \ref{table:GRID2SQUAREvalue}, \ref{table:GRID2LONGvalue} and \ref{table:GRID2WIDEvalue}  are tight.
We report here only results for large instances with the number of arcs ranging from $1352$ to $2646$.
However, we could solve even larger instances but the computation time would exceed one hour.
Note also that for several instances the GL bound is trivial, i.e., equal to zero.

The reason for being able to solve large {\sc grid}2 instances could be explained as follows.
As the costs of the arcs are independent, a path with longer length is expected to  have a higher cost.
Therefore, an optimal path tends to be the path with a smaller length.
Indeed,  we observe that  the length of the optimal path for any test instance reported
in Table  \ref{table:GRID2SQUAREvalue}, \ref{table:GRID2LONGvalue} or \ref{table:GRID2WIDEvalue}
is longer for at most three arcs from the  minimal path length $q+1$.

\item[(iii)]  Tables \ref{table:GRID2SQUAREvalue2}, \ref{table:GRID2LONGvalue2} and \ref{table:GRID2WIDEvalue2} present results for the {\sc grid}3 instances.
For each different size of  $m$, we generate four {\sc grid}3 instances with $d'= 0,0, 0.1, 0.1$, respectively,  and $d = 0.5$ fixed.
Small $d'$ enables that a path with a length longer than  $q+1$ is more likely to be an optimal path.
This results with more difficult instances than the {\sc grid}2  instances.
Consequently we were only able to compute lower bounds for instances of  up to $2000$ arcs in a reasonable amount of time.
We remark that for the smaller size instances than those presented in the tables, the SDP lower bounds are mostly tight.

The upper bounds reported in Tables \ref{table:GRID2SQUAREvalue2}, \ref{table:GRID2LONGvalue2} and \ref{table:GRID2WIDEvalue2} are obtained by solving the linear programming problem \eqref{admm_lpupper},
or by using simulated annealing. In particular, we write down the better among these two.
The reason that we also use simulated annealing is that the  SDP lower bounds are sometimes not strong enough.
For the test instances for which the SDP bound is tight, we observe that  the length of the optimal path here might be longer up to seventeen more arcs than $q+1$.

\item[(iv)]
 In Table \ref{table:GRID4SQUAREvalue} we report results for the {\sc grid}4 instances. For each size, four instances are generated with $d = 0.2,0.2,0.8,0.8$, respectively.
Similar to {\sc grid}2 and {\sc grid}3 instances, the optimal path tends to have shorter length. Consequently the problem is easy to solve when the cost matrix is dense.
In fact, removing all the arcs of the form $(v_{i,j},v_{i-1,j})$ or $(v_{i,j},v_{i,j-1})$ does not change the optimal value for all the tested instances with high density $d = 0.8$.
To the contrary, an instance with low density $d = 0.2$ is much harder to solve, and none of the lower bounds is tight in this case.
Upper bounds in Table \ref{table:GRID4SQUAREvalue} are obtained by solving \eqref{admm_lpupper}.

\item[(v)]
We report numerical results  for the {\sc par}-{\sc k} instances in Table \ref{table:KpartiteEvalue} and \ref{table:KpartiteTime}.
For each $K$, we generate four instances with $d = 0.8$.
The relaxation $SDP_{LS+}$ provides trivial lower bounds with negative values,
while $SDP_{NLS}$ remains  strong.  In particular,   $SDP_{NLS}$ provides  optimal values for all tested instances with $m \leq 720$.
Note also that RBB and GL give weak bounds for these dense instances.
Upper bounds in Table \ref{table:KpartiteEvalue} are obtained by solving \eqref{admm_lpupper}.

\end{enumerate}

It is also worth mentioning that we tested QSPP instances on the double-directed grid graphs, where
quadratic costs are given as reload costs, see \cite{galbiati2014minimum, gourves2009minimum}.
In particular,  each arc is colored by one of the given $c$ colors and there is  no interaction cost between arcs with the same color.
For so generated  instances, the  GL and RBB bounds equal to the bound obtained by solving the standard shortest path problem using the linear cost.
However, our strongest SDP relaxation provides tight bounds for large instances.

\medskip

To summarize, we present numerical results for many different types of the QSPP instances whose sizes vary from 220 to 2646 arcs.
Since for smaller instances we mostly obtain tight bounds, we do not present those results.
Our results show that the SDP bounds together with the ADMM  make a powerful combination for the computations of strong bounds for the QSPP.
Finally, we show that adding redundant constraints to the SDP relaxation helps to improve the performance of the ADMM.
We  exploit this to develop  an efficient  branch-and-bound algorithm for solving the QSPP to optimality.

\medskip
\medskip
\noindent
{\bf Acknowledgements.}  The authors would like to thank Henry Wolkowicz and Lieven Vandenberghe for useful discussions on the ADMM and facial reduction.

\begin{table}[H]
	\centering
	{\footnotesize
		\begin{tabular}{|cc|ccc|cc|cc|}
			\hline
$n$ & $m$ & $gl$ & $rbb$ & $sdp_{ls+}$ & $sdp_{nls}$ & $sdp_{nls}^{ub}$ & $BnB_{root}^{UB}$ & $BnB(opt)$ \\\hline
121  & 220  & 2  & 21.55 & 132.90 & 205.72 & 206  & 200.28 & 206 \\
121  & 220  & 11  & 25.78 & 128.60 & 181  & 181  & 179.56 & 181 \\
121  & 220  & 750  & 978.63 & 1319.62 & 1374.96 & 1375  & 1373.31 & 1375 \\
121  & 220  & 740  & 950.50 & 1277.39 & 1323.72 & 1324  & 1319.31 & 1324 \\
&    &    &    &    &    &    &    &   \\
144  & 264  & 3  & 27.59 & 173.63 & 248  & 248  & 244.33 & 248 \\
144  & 264  & 17  & 36.50 & 151.08 & 221  & 221  & 217.97 & 221 \\
144  & 264  & 867  & 1161.75 & 1552.80 & 1589  & 1589  & 1585.45 & 1589 \\
144  & 264  & 898  & 1192.38 & 1591.41 & 1611  & 1611  & 1608.93 & 1611 \\
&    &    &    &    &    &    &    &   \\
169  & 312  & 14  & 46.14 & 197.76 & 298.42 & 299  & 297.74 & 299 \\
169  & 312  & 0  & 27.33 & 190.14 & 263  & 263  & 257.03 & 263 \\
169  & 312  & 1042  & 1376.38 & 1902.88 & 2001.62 & 2004  & 1990.20 & 2004 \\
169  & 312  & 1066  & 1399  & 1917.76 & 2064.88 & 2065  & 2041.45 & 2065 \\
&    &    &    &    &    &    &    &   \\
196  & 364  & 3  & 18.14 & 177.12 & 331  & 331  & 324.81 & 331 \\
196  & 364  & 1  & 21.50 & 211.56 & 365.93 & 366  & 357.43 & 366 \\
196  & 364  & 1227  & 1631.88 & 2242.57 & 2328  & 2328  & 2322.88 & 2328 \\
196  & 364  & 1210  & 1617.50 & 2262.75 & 2338  & 2338  & 2336.72 & 2338 \\
&    &    &    &    &    &    &    &   \\
225  & 420  & 13  & 32.19 & 226.66 & 382  & 382  & 369.06 & 382 \\
225  & 420  & 11  & 58.14 & 267.89 & 450.86 & 459  & 435.96 & 459 \\
225  & 420  & 654  & 1088.50 & 1741.08 & 1955.95 & 1957  & 1943.75 & 1956 \\
225  & 420  & 1447  & 1938.50 & 2633.71 & 2794.97 & 2795  & 2776  & 2795 \\
&    &    &    &    &    &    &    &   \\
256  & 480  & 9  & 40.38 & 257.66 & 457  & 457  & 443.35 & 457 \\
256  & 480  & 5  & 41.34 & 297.91 & 489  & 489  & 475.56 & 489 \\
256  & 480  & 661  & 1154.81 & 1893.53 & 2163.13 & 2165  & 2135.41 & 2165 \\
256  & 480  & 1592  & 2118.75 & 3012.44 & 3267.70 & 3276  & 3231.93 & 3276 \\
&    &    &    &    &    &    &    &   \\
289  & 544  & 11  & 44.02 & 307.21 & 543.80 & 544  & 529.74 & 544 \\
289  & 544  & 3  & 59.28 & 341.04 & 552.43 & 553  & 540.22 & 553 \\
289  & 544  & 804  & 1367.06 & 2244.03 & 2515.76 & 2516  & 2495.51 & 2516 \\
289  & 544  & 1794  & 2329.38 & 3375.07 & 3676  & 3676  & 3657.65 & 3676 \\
&    &    &    &    &    &    &    &   \\
324  & 612  & 3  & 39.90 & 322.10 & 616.44 & 638  & 601.15 & 622 \\
324  & 612  & 13  & 63.64 & 359.33 & 649  & 649  & 638.26 & 649 \\
324  & 612  & 858  & 1555.19 & 2496.70 & 2861.30 & 2863  & 2824.80 & 2863 \\
324  & 612  & 1954  & 2645.63 & 3790.17 & 4147.71 & 4149  & 4113.88 & 4149 \\
&    &    &    &    &    &    &    &   \\
361  & 684  & 7  & 38.66 & 355.35 & 682.39 & 786  & 669.61 & 715 \\
361  & 684  & 6  & 32.92 & 342.55 & 680.55 & 681  & 663.38 & 681 \\
361  & 684  & 939  & 1670.13 & 2845.50 & 3274.01 & 3477  & 3246.62 & 3307 \\
361  & 684  & 2260  & 3025.50 & 4292.52 & 4662  & 4662  & 4632.17 & 4662 \\
&    &    &    &    &    &    &    &   \\
400  & 760  & 8  & 32.87 & 393.38 & 746.53 & 747  & 730.06 & 747 \\
400  & 760  & 5  & 42.10 & 428.69 & 809.13 & 858  & 793.11 & 837 \\
400  & 760  & 1052  & 1902.31 & 3109.75 & 3580  & 3580  & 3544.68 & 3580 \\
400  & 760  & 2465  & 3381.13 & 4773.37 & 5224.91 & 5226  & 5184.97 & 5226 \\  \hline
	\end{tabular} } \caption{{\sc grid}1-{\sc square}: bounds and optimal values} \label{table:LBdirectedGridgraphs}
\end{table}

\begin{table}[H]
	\centering
	{\footnotesize
		\begin{tabular}{|cc|c|cc|cc|cc|cc|}
			\hline
$n$ & $m$ & $gl(s)$ & $rbb(s)$ & $rbb(it)$ & $sdp_{ls+}(s)$ & $sdp_{ls+}(it)$ & $sdp_{nls}(s)$ & $sdp_{nls}(it)$ & $BnB(s)$ & $BnB(n)$ \\\hline
121  & 220  & 0.17 & 1.97 & 9  & 7.41 & 3323  & 1.43 & 705  & 4.70 & 17 \\
121  & 220  & 0.15 & 1.56 & 7  & 7.46 & 3381  & 1.35 & 692  & 3.74 & 5 \\
121  & 220  & 0.14 & 0.89 & 4  & 7.32 & 3425  & 1.49 & 785  & 3.94 & 9 \\
121  & 220  & 0.14 & 0.89 & 4  & 7.33 & 3294  & 0.82 & 414  & 4.10 & 9 \\
&    &    &    &    &    &    &    &    &    &   \\
144  & 264  & 0.20 & 1.79 & 6  & 11.16 & 3430  & 2.71 & 932  & 4.88 & 7 \\
144  & 264  & 0.20 & 1.49 & 5  & 11.61 & 3580  & 1.59 & 544  & 5.06 & 9 \\
144  & 264  & 0.20 & 1.23 & 4  & 10.45 & 3232  & 0.87 & 299  & 6.52 & 37 \\
144  & 264  & 0.20 & 1.21 & 4  & 10.88 & 3263  & 0.90 & 305  & 4.79 & 5 \\
&    &    &    &    &    &    &    &    &    &   \\
169  & 312  & 0.27 & 2.87 & 7  & 16.59 & 3454  & 5.70 & 1332  & 7.18 & 11 \\
169  & 312  & 0.26 & 3.19 & 8  & 17.17 & 3633  & 3.37 & 779  & 6.58 & 11 \\
169  & 312  & 0.26 & 1.66 & 4  & 15.80 & 3346  & 2.93 & 674  & 7.25 & 45 \\
169  & 312  & 0.26 & 1.67 & 4  & 15.93 & 3381  & 17.78 & 4027  & 10.39 & 83 \\
&    &    &    &    &    &    &    &    &    &   \\
196  & 364  & 0.36 & 4.22 & 7  & 32.35 & 3518  & 12.38 & 1434  & 8.88 & 11 \\
196  & 364  & 0.37 & 5.40 & 9  & 33.47 & 3593  & 33.17 & 3845  & 8.75 & 11 \\
196  & 364  & 0.35 & 2.45 & 4  & 30.89 & 3396  & 3.18 & 371  & 11.80 & 39 \\
196  & 364  & 0.34 & 2.42 & 4  & 31.51 & 3398  & 3.61 & 427  & 5.04 & 3 \\
&    &    &    &    &    &    &    &    &    &   \\
225  & 420  & 0.46 & 4.89 & 6  & 47.26 & 3590  & 43.26 & 3374  & 18.05 & 31 \\
225  & 420  & 0.46 & 5.74 & 7  & 46.06 & 3662  & 51.24 & 4097  & 22.01 & 29 \\
225  & 420  & 0.44 & 4.15 & 5  & 43.56 & 3492  & 9.24 & 777  & 16.19 & 25 \\
225  & 420  & 0.48 & 3.98 & 4  & 46.07 & 3423  & 17.97 & 1405  & 18.63 & 43 \\
&    &    &    &    &    &    &    &    &    &   \\
256  & 480  & 0.67 & 9.48 & 7  & 65.97 & 3650  & 42.34 & 2660  & 21.01 & 19 \\
256  & 480  & 0.62 & 9.07 & 8  & 60.90 & 3583  & 29.24 & 1814  & 28.31 & 19 \\
256  & 480  & 0.59 & 5.84 & 5  & 64.80 & 3772  & 60.79 & 3799  & 36.42 & 81 \\
256  & 480  & 0.58 & 4.67 & 4  & 59.37 & 3479  & 63.45 & 3707  & 27.70 & 67 \\
&    &    &    &    &    &    &    &    &    &   \\
289  & 544  & 0.99 & 15.93 & 7  & 95.72 & 3726  & 89.33 & 3594  & 30.20 & 11 \\
289  & 544  & 1  & 16.10 & 7  & 99.77 & 3771  & 35.50 & 1432  & 54.73 & 51 \\
289  & 544  & 0.94 & 10.43 & 5  & 93.10 & 3650  & 31.43 & 1326  & 58.73 & 53 \\
289  & 544  & 0.86 & 7.37 & 4  & 80.23 & 3492  & 72.88 & 3268  & 29.83 & 63 \\
&    &    &    &    &    &    &    &    &    &   \\
324  & 612  & 1.12 & 20.67 & 8  & 119.62 & 3841  & 127.83 & 4286  & 69.09 & 61 \\
324  & 612  & 1.12 & 20.18 & 8  & 118.87 & 3808  & 110.88 & 3733  & 66.73 & 35 \\
324  & 612  & 1.04 & 12.84 & 5  & 116.82 & 3767  & 121.60 & 4097  & 56.22 & 97 \\
324  & 612  & 1.06 & 10.40 & 4  & 109.41 & 3535  & 110.56 & 3722  & 44.13 & 63 \\
&    &    &    &    &    &    &    &    &    &   \\
361  & 684  & 1.49 & 28.18 & 8  & 158.36 & 3868  & 168.60 & 4320  & 131.35 & 93 \\
361  & 684  & 1.51 & 24.56 & 7  & 155.70 & 3858  & 152.63 & 3929  & 88.06 & 21 \\
361  & 684  & 1.42 & 17.88 & 5  & 151.74 & 3716  & 157.10 & 3965  & 133.96 & 55 \\
361  & 684  & 1.63 & 16.73 & 4  & 160.54 & 3557  & 62.77 & 1414  & 75.58 & 71 \\
&    &    &    &    &    &    &    &    &    &   \\
400  & 760  & 2.20 & 41.17 & 8  & 221.28 & 4067  & 76.24 & 1536  & 138.45 & 55 \\
400  & 760  & 1.88 & 37.88 & 8  & 237.60 & 4072  & 249.98 & 4349  & 198.43 & 67 \\
400  & 760  & 1.76 & 23.86 & 5  & 201.16 & 3866  & 191.43 & 3833  & 114.30 & 73 \\
400  & 760  & 1.79 & 19.08 & 4  & 192.58 & 3640  & 200.97 & 3980  & 164.86 & 144 \\ \hline
	\end{tabular}} \caption{{\sc grid}1-{\sc square}:  running times and iterations} \label{table:LBdirectedGridgraphsTime}
\end{table}

\begin{table}[H]
	\centering
	{\footnotesize
		\begin{tabular}{|cc|cccc|ccccc|}
			\hline
			$n$ & $m$ & $gl$ & $rbb$ & $sdp_{nls}$ & $sdp_{nls}^{UB}$ & $gl(s)$ & $rbb(s)$ & $rbb(it)$ & $sdp_{nls}(s)$ & $sdp_{nls}(it)$ \\\hline
			678  & 1352  & 0  & 15.30 & 598  & 598  & 8.90 & 337.96 & 13  & 157.04 & 626 \\
			678  & 1352  & 0  & 16.67 & 564  & 564  & 8.78 & 260.26 & 10  & 398.39 & 1628 \\
			678  & 1352  & 1843  & 2433.56 & 3001  & 3001  & 8.66 & 156.87 & 6  & 168.63 & 683 \\
			678  & 1352  & 1918  & 2451.09 & 2988  & 2988  & 8.70 & 157.03 & 6  & 129.40 & 508 \\
			&    &    &    &    &    &    &    &    &    &   \\
			731  & 1458  & 0  & 8.49 & 587  & 587  & 10.93 & 324.82 & 10  & 2209.94 & 7563 \\
			731  & 1458  & 0  & 24.25 & 625  & 625  & 10.94 & 355.14 & 11  & 875.70 & 2986 \\
			731  & 1458  & 1980  & 2492.94 & 3127  & 3127  & 10.75 & 192.05 & 6  & 184.63 & 601 \\
			731  & 1458  & 1940  & 2534.22 & 3267  & 3267  & 10.69 & 192.25 & 6  & 189.17 & 613 \\   \hline
	\end{tabular}} \caption{{\sc grid}2-{\sc square}: bounds, running times, iterations}  \label{table:GRID2SQUAREvalue}
\end{table}

\begin{table}[H]
	\centering
	{\footnotesize
		\begin{tabular}{|cc|cccc|ccccc|}
			\hline
			$n$ & $m$ & $gl$ & $rbb$ & $sdp_{nls}$ & $sdp_{nls}^{UB}$ & $gl(s)$ & $rbb(s)$ & $rbb(it)$ & $sdp_{nls}(s)$ & $sdp_{nls}(it)$ \\\hline
			1158  & 2261  & 233  & 1020.13 & 4648  & 4648  & 38.32 & 738.20 & 6  & 2516.16 & 3090 \\
			1158  & 2261  & 221  & 956.78 & 4641  & 4641  & 38.05 & 735.49 & 6  & 2898.22 & 3543 \\
			1158  & 2261  & 13643  & 16898.06 & 19950  & 19950  & 38.07 & 747.63 & 6  & 744.92 & 866 \\
			1158  & 2261  & 13801  & 17255.94 & 20423  & 20423  & 38.04 & 748.45 & 6  & 1373.16 & 1631 \\
			&    &    &    &    &    &    &    &    &    &   \\
			1298  & 2538  & 259  & 1093.03 & 5224  & 5224  & 52.24 & 1033.70 & 6  & 2578.47 & 2327 \\
			1298  & 2538  & 288  & 1152.31 & 5243  & 5243  & 52.15 & 1036.49 & 6  & 2661.94 & 2392 \\
			1298  & 2538  & 15479  & 19311.78 & 22971  & 22971  & 52.48 & 1052.25 & 6  & 4387.72 & 3910 \\
			1298  & 2538  & 15424  & 19186.53 & 22755  & 22755  & 52.52 & 1051.36 & 6  & 1675.82 & 1450 \\ \hline
	\end{tabular}} \caption{{\sc grid}2-{\sc long}: bounds, running times, iterations} \label{table:GRID2LONGvalue}
\end{table}

\begin{table}[H]
	\centering
	{\footnotesize
		\begin{tabular}{|cc|cccc|ccccc|}
			\hline
			$n$ & $m$ & $gl$ & $rbb$ & $sdp_{nls}$ & $sdp_{nls}^{UB}$ & $gl(s)$ & $rbb(s)$ & $rbb(it)$ & $sdp_{nls}(s)$ & $sdp_{nls}(it)$ \\\hline
			1158  & 2363  & 0  & 2.84 & 204  & 204  & 39.96 & 1595.04 & 13  & 2817.37 & 3046 \\
			1158  & 2363  & 0  & 2.57 & 186  & 186  & 40.67 & 1220.44 & 10  & 4629.67 & 5016 \\
			1158  & 2363  & 782  & 980.56 & 1240  & 1240  & 39.19 & 716.12 & 6  & 1676.66 & 1759 \\
			1158  & 2363  & 785  & 981.91 & 1236  & 1236  & 39.32 & 721.11 & 6  & 1085.42 & 1125 \\
			&    &    &    &    &    &    &    &    &    &   \\
			1298  & 2646  & 0  & 0.52 & 241  & 241  & 55.62 & 1206.43 & 7  & 2733.47 & 2200 \\
			1298  & 2646  & 0  & 2.57 & 206  & 206  & 55.69 & 2918.56 & 17  & 685.66 & 525 \\
			1298  & 2646  & 906  & 1143.25 & 1406  & 1406  & 54.19 & 1004.12 & 6  & 1257.34 & 967 \\
			1298  & 2646  & 884  & 1116.31 & 1385  & 1385  & 54.09 & 1004.72 & 6  & 836.93 & 627 \\  \hline
	\end{tabular}} \caption{{\sc grid}2-{\sc wide}: bounds, running times, iterations} \label{table:GRID2WIDEvalue}
\end{table}

\begin{table}[H]
\centering
{\footnotesize
\begin{tabular}{|cc|cccc|ccccc|}
\hline
$n$ & $m$ & $gl$ & $rbb$ & $sdp_{nls}$ & $ub$ & $gl(s)$ & $rbb(s)$ & $rbb(it)$ & $sdp_{nls}(s)$ & $sdp_{nls}(it)$ \\\hline
363  & 722  & 38  & 64.88 & 576.43 & 582  & 1.69 & 21.13 & 5  & 342.67 & 7304 \\
363  & 722  & 36  & 51.13 & 577.70 & 579  & 1.56 & 20.40 & 5  & 345.94 & 7400 \\
363  & 722  & 54  & 118.97 & 762  & 762  & 1.64 & 25.91 & 6  & 19.83 & 421 \\
363  & 722  & 36  & 111.14 & 768  & 768  & 1.65 & 29.79 & 7  & 33.94 & 729 \\
&    &    &    &    &    &    &    &    &    &   \\
402  & 800  & 27  & 32.25 & 601  & 601  & 1.98 & 21.74 & 4  & 80.55 & 1397 \\
402  & 800  & 26  & 38.88 & 638.45 & 671  & 1.94 & 27.88 & 5  & 400.87 & 6891 \\
402  & 800  & 43  & 113.72 & 882.35 & 888  & 2.04 & 39.54 & 7  & 504.74 & 8737 \\
402  & 800  & 37  & 88.75 & 878.70 & 879  & 2.05 & 39.60 & 7  & 274.39 & 4774 \\    \hline
	\end{tabular}} \caption{{\sc grid}3-{\sc square}: bounds, running times, iterations} \label{table:GRID2SQUAREvalue2}
\end{table}

\begin{table}[H]
	\centering
	{\footnotesize
		\begin{tabular}{|cc|cccc|ccccc|}
			\hline
$n$ & $m$ & $gl$ & $rbb$ & $sdp_{nls}$ & $ub$ & $gl(s)$ & $rbb(s)$ & $rbb(it)$ & $sdp_{nls}(s)$ & $sdp_{nls}(it)$ \\\hline
902  & 1755  & 2548  & 4445.19 & 8297  & 8297  & 17.61 & 280.89 & 5  & 1344.58 & 3634 \\
902  & 1755  & 2625  & 4145.44 & 8347.39 & 8384  & 17.11 & 278.69 & 5  & 4442.90 & 11946 \\
902  & 1755  & 2915  & 4885.81 & 9010  & 9010  & 17.99 & 287.92 & 5  & 844.73 & 2262 \\
902  & 1755  & 2913  & 4848.56 & 9059  & 9059  & 18.24 & 288.45 & 5  & 1572.34 & 4198 \\
&    &    &    &    &    &    &    &    &    &   \\
1026  & 2000  & 2998  & 4812.06 & 9450.67 & 9697  & 25.32 & 420.15 & 5  & 7614.34 & 14625 \\
1026  & 2000  & 3033  & 4992.44 & 9431  & 9431  & 25.03 & 419.66 & 5  & 3414.82 & 6593 \\
1026  & 2000  & 3108  & 5565.56 & 10251  & 10251  & 26.24 & 428.71 & 5  & 3555.13 & 7115 \\
1026  & 2000  & 3121  & 5296.31 & 10261.56 & 10341  & 26.22 & 426.64 & 5  & 8473.91 & 16472 \\  \hline
	\end{tabular}} \caption{{\sc grid}3-{\sc long}: bounds, running times, iterations} \label{table:GRID2LONGvalue2}
\end{table}

\begin{table}[H]
	\centering
	{\footnotesize
		\begin{tabular}{|cc|cccc|ccccc|}
			\hline
$n$ & $m$ & $gl$ & $rbb$ & $sdp_{nls}$ & $ub$ & $gl(s)$ & $rbb(s)$ & $rbb(it)$ & $sdp_{nls}(s)$ & $sdp_{nls}(it)$ \\\hline
678  & 1313  & 2243  & 3472.81 & 6294  & 6294  & 8.72 & 146.23 & 6  & 553.25 & 3040 \\
678  & 1313  & 2150  & 3479.13 & 6207  & 6207  & 7.74 & 143.66 & 6  & 505.02 & 2783 \\
678  & 1313  & 2271  & 3899.16 & 6769  & 6769  & 8.26 & 147.67 & 6  & 619.16 & 3471 \\
678  & 1313  & 2336  & 3926.78 & 6853  & 6853  & 8.46 & 147.64 & 6  & 402.67 & 2180 \\
   &    &    &    &    &    &    &    &    &    &   \\
786  & 1526  & 2477  & 4006.72 & 7325.83 & 7352  & 12.25 & 227.85 & 6  & 6670.28 & 25000 \\
786  & 1526  & 2311  & 3883.09 & 7036  & 7036  & 12.18 & 227.35 & 6  & 594.11 & 2310 \\
786  & 1526  & 2532  & 4312.78 & 7825  & 7825  & 12.81 & 231.30 & 6  & 2403.64 & 9253 \\
786  & 1526  & 2581  & 4394.06 & 7883  & 7883  & 12.81 & 231.37 & 6  & 867.23 & 3290 \\  \hline
	\end{tabular}} \caption{{\sc grid}3-{\sc wide}: bounds, running times, iterations} \label{table:GRID2WIDEvalue2}
\end{table}

\begin{table}[H]
	\centering
	{\footnotesize
\begin{tabular}{|cc|c|cc|ccc|}
\hline
$n$ & $m$ & $gl$ & $sdp_{nls}$ & $sdp_{nls}^{UB}$ & $gl(s)$ & $sdp_{nls}(s)$ & $sdp_{nls}(it)$ \\\hline
121  & 440  & 0  & 149.02 & 265  & 0.26 & 83.69 & 4118 \\
121  & 440  & 0  & 132.23 & 216  & 0.24 & 78  & 3874 \\
121  & 440  & 606  & 1375  & 1375  & 0.34 & 19.16 & 972 \\
121  & 440  & 579  & 1323.72 & 1324  & 0.35 & 7.90 & 398 \\
&    &    &    &    &    &    &   \\
144  & 528  & 0  & 166.33 & 264  & 0.46 & 139.44 & 4360 \\
144  & 528  & 0  & 169.06 & 250  & 0.38 & 130.39 & 4072 \\
144  & 528  & 682  & 1589  & 1589  & 0.50 & 8.67 & 267 \\
144  & 528  & 700  & 1611  & 1611  & 0.50 & 9.41 & 289 \\
&    &    &    &    &    &    &   \\
169  & 624  & 4  & 219.01 & 450  & 0.53 & 203.16 & 4217 \\
169  & 624  & 0  & 207.01 & 325  & 0.52 & 201.29 & 4170 \\
169  & 624  & 816  & 2004  & 2004  & 0.71 & 32.10 & 662 \\
169  & 624  & 811  & 2064.91 & 2065  & 0.74 & 236.86 & 4889 \\
&    &    &    &    &    &    &   \\
196  & 728  & 0  & 212.31 & 506  & 0.77 & 301.13 & 4274 \\
196  & 728  & 0  & 241.46 & 596  & 0.82 & 301.07 & 4263 \\
196  & 728  & 948  & 2328  & 2328  & 1.48 & 25.17 & 350 \\
196  & 728  & 941  & 2338  & 2338  & 1.15 & 27.01 & 374 \\ \hline
\end{tabular}} \caption{{\sc grid}4: bounds, running times, iterations} \label{table:GRID4SQUAREvalue}
\end{table}

\begin{table}[H]
	\centering
	{\footnotesize
\begin{tabular}{|ccc|cc|c|cc|}
	\hline
	$K$ & $n$ & $m$ & $gl$ & $rbb$ & $sdp_{ls+}$ & $sdp_{nls}$ & $sdp_{nls}^{UB}$ \\\hline
	9  & 65  & 504  & 10  & 40.91 & -265.04 & 115.85 & 116 \\
	9  & 65  & 504  & 10  & 40.13 & -266.50 & 113  & 113 \\
	9  & 65  & 504  & 9  & 40.19 & -272.79 & 96  & 96 \\
	9  & 65  & 504  & 8  & 30.84 & -275  & 101  & 101 \\
	&    &    &    &    &    &    &   \\
	10  & 82  & 720  & 3  & 40.80 & -450.79 & 139.99 & 140 \\
	10  & 82  & 720  & 7  & 33.25 & -446.06 & 125  & 125 \\
	10  & 82  & 720  & 4  & 36.59 & -446.21 & 136  & 136 \\
	10  & 82  & 720  & 10  & 37.47 & -444.82 & 137  & 137 \\
	&    &    &    &    &    &    &   \\
	11  & 101  & 990  & 3  & 34.20 & -674.08 & 160.12 & 270 \\
	11  & 101  & 990  & 8  & 34.95 & -667.69 & 168.09 & 210 \\
	11  & 101  & 990  & 5  & 34.50 & -667.78 & 164.21 & 301 \\
	11  & 101  & 990  & 6  & 34.44 & -668.19 & 165.12 & 239 \\     \hline
	\end{tabular}} 	\caption{{\sc par}-{\sc k}: bounds} \label{table:KpartiteEvalue}
\end{table}

\begin{table}[H]
	\centering
	{\footnotesize
\begin{tabular}{|ccc|c|cc|cc|cc|}
	\hline
$K$ & $n$ & $m$ & $gl(s)$ & $rbb(s)$ & $rbb(it)$ & $sdp_{ls+}(s)$ & $sdp_{ls+}(it)$ & $sdp_{nls}(s)$ & $sdp_{nls}(it)$ \\\hline
9  & 65  & 504  & 0.43 & 4.18 & 6  & 151.79 & 3605  & 40.28 & 953 \\
9  & 65  & 504  & 0.44 & 4.11 & 6  & 149.55 & 3515  & 64.35 & 1548 \\
9  & 65  & 504  & 0.45 & 4.16 & 6  & 153.17 & 3647  & 20.51 & 466 \\
9  & 65  & 504  & 0.44 & 4.23 & 6  & 152.13 & 3617  & 24.65 & 559 \\
&    &    &    &    &    &    &    &    &   \\
10  & 82  & 720  & 0.88 & 10.76 & 7  & 399.62 & 3987  & 284.72 & 2720 \\
10  & 82  & 720  & 0.90 & 9.15 & 6  & 389.82 & 3832  & 111.59 & 1028 \\
10  & 82  & 720  & 0.94 & 9.09 & 6  & 372.30 & 3713  & 235.01 & 2203 \\
10  & 82  & 720  & 0.92 & 9.14 & 6  & 375.81 & 3713  & 233.67 & 2256 \\
&    &    &    &    &    &    &    &    &   \\
11  & 101  & 990  & 1.63 & 19.99 & 7  & 980.90 & 4158  & 1580.83 & 6108 \\
11  & 101  & 990  & 1.67 & 20.12 & 7  & 936.08 & 3962  & 1598.11 & 6163 \\
11  & 101  & 990  & 1.69 & 18.36 & 6  & 953.89 & 3985  & 1521.10 & 5859 \\
11  & 101  & 990  & 1.66 & 20.67 & 7  & 965.39 & 4081  & 1532.04 & 5904 \\  \hline
\end{tabular}}\caption{ {\sc par}-{\sc k}:  running times, iterations} \label{table:KpartiteTime}
\end{table}

\newpage

\begin{thebibliography}{10}

\bibitem{boyd2011distributed}
S. Boyd, N. Parikh, E. Chu, B. Peleato, J. Eckstein.
Distributed optimization and statistical learning via the alternating direction method of multipliers.
 Foundations and Trends{\textregistered} in Machine Learning,  3(1):1--122, 2011.

\bibitem{buchheim2015quadratic}
C. Buchheim, E. Traversi.
Quadratic 0--1 optimization using separable underestimators.
Technical Report, Optimization Online, 2015.

\bibitem{biggs1993algebraic}
N. Biggs. {\em Algebraic graph theory}. Cambridge university press, 1993.
	
\bibitem{condat2014fast}
L. Condat. Fast projection onto the simplex and the ${l}_{1}$ ball.
{\em Math. Programming},  158(1):575--585, 2016.

\bibitem{cplex20007}
Cplex, ILOG. 7.0 Reference Manual. ILOG CPLEX Division, Incline Village, NV. 2000.

\bibitem{DrusWolko:17}
D. Drusvyatskiy, H. Wolkowicz.
The many faces of degeneracy in conic optimization. Preprint 2017, arXiv:1706.03705.

\bibitem{galbiati2008complexity}
G. Galbiati. The complexity of a minimum reload cost diameter problem.
{\em Discrete Appl. Math.}, 156(18):3494-3497, 2008.

\bibitem{galbiati2014minimum}
G. Galbiati, S. Gualandi, F. Maffioli.
On minimum reload cost cycle cover. {\em Discrete Appl. Math.}, 164(1):112–120, 2014.

\bibitem{gamvros2006satellite}
I. Gamvros.
Satellite network design, optimization and management. PhD thesis, University of Maryland, 2006.


\bibitem{gourves2009minimum}
L. Gourv{\`e}s, A. Lyra, C. Martinhon and J. Monnot.
The minimum reload $s$-$t$ path/trail/walk problems.
SOFSEM 2009: Theory and Practice of Computer Science, Springer, p.~621--632, 2009.

\bibitem{held1974validation}
M. Held,  P. Wolfe, H.P. Crowder.
Validation of subgradient optimization.
{\em  Math. Programming}, 6(1):62--88, 1974.

\bibitem{hu2016qspp}
H. Hu, R. Sotirov, Special cases of the quadratic shortest path problem. arXiv:1611.07682  [math.OC].

\bibitem{kovacs2015minimum}
P. Kov\'acs.  Minimum-cost flow algorithms: An experimental evaluation.
{\em Optimization Methods and Software}, 30(1):94--127, 2015.

\bibitem{liao2016QAPbbadmm}
Z. Liao.
Branch and bound via the alternating direction method of multipliers for the quadratic assignment problem.
Master Thesis, University of Waterloo, 2016.

\bibitem{mosek2010mosek}
MOSEK, Aps. The MOSEK optimization software. Online at http://www.mosek.com.
vol.~54, 2010.

\bibitem{murakami1997comparative}
K. Murakami,  H.S. Kim. Comparative study on restoration schemes of survivable ATM networks,
INFOCOM'97. Sixteenth Annual Joint Conference of the IEEE Computer and Communications Societies.
Driving the Information Revolution., Proceedings IEEE, vol.~1:345--352, 1997.

\bibitem{NieWu}
Y.M. Nie, X. Wu.
Reliable a priori shortest path problem with limited spatial and temporal dependencies.
In: W.H.K. Lam, S.C. Wong, H.K. Lo (eds.),  {\em Transportation and Traffic Theory 2009: Golden Jubilee}, 169--195,  2009.

\bibitem{oliveira2015admm}
D.E. Oliveira,  H. Wolkowicz,  Y. Xu.
ADMM for the SDP relaxation of the QAP, arXiv preprint arXiv:1512.05448, 2015.

\bibitem{povh2006boundary}
J. Povh, F. Rendl,  A. Wiegele.
A boundary point method to solve semidefinite programs. {\em Computing}, 78(3):277--286, 2006.

\bibitem{rostami2016quadratic}
B. Rostami, A. Chassein, M. Hopf,  D. Frey,  C. Buchheim, F. Malucelli, M. Goerigk.
The quadratic shortest path problem: complexity, approximability, and solution methods, Optimization online, 2016.	

\bibitem{frey2015quadratic}
B. Rostami, F. Malucelli, D. Frey, C. Buchheim.
On the quadratic shortest path problem.
In: E. Bampis (ed.) Experimental Algorithms, Lecture Notes in Computer Science, vol.~9125,
Springer International Publishing,   379--390,  2015.

\bibitem{Sen:01}
S.  Sen, R. Pillai, S. Joshi, A.K. Rathi.
A mean-variance model for route guidance in advanced traveler information systems.
{\em Transportation Science}, 35(1):37--49, 2001.

\bibitem{sivakumar1994variance}
R.A.  Sivakumar, R.  Batta.
The variance-constrained shortest path problem.
{\em Transportation Science},  28(4):309--316, 1994.

\bibitem{tuncel}
L. Tun\c{c}el.
On the Slater condition for the SDP relaxations of nonconvex sets. {\em Oper. Res. Lett.}, 29(4):181-–186, 2001.

\bibitem{wen2010alternating}
Z. Wen, D. Goldfarb, and  W. Yin.
Alternating direction augmented Lagrangian methods for semidefinite programming.
{\em Math. Program. Comput.}, 2(3):203--230, 2010.

\bibitem{wirth2001reload}
H.C. Wirth, J.Steffan.
Reload cost problems: minimum diameter spanning tree.
{\em Discrete Appl. Math.}, 113:73-–85, 2001.

\bibitem{zhao1998semidefinite}
Q. Zhao, S.E. Karisch, F. Rendl, H. Wolkowicz.
Semidefinite programming relaxations for the quadratic assignment problem.
{\em J. Comb. Optim.}, 2:71--109, 1998.

\end{thebibliography}
\end{document}